\documentclass[10pt,a4paper]{article}     
\usepackage{amsmath,amssymb,amsthm,color,enumerate,fancyhdr,float,fullpage,graphics,graphicx,longtable,mathrsfs,multirow,pgf,setspace,subfig,tikz,url} 
\usepackage[font=small,labelfont=bf]{caption}
\usepackage[toc,page]{appendix}
\usepackage{hyperref}
\hypersetup{colorlinks=true,linkcolor=blue,citecolor=red}

\usepackage{cleveref}

\allowdisplaybreaks  
\usetikzlibrary{arrows}

\def\Eb{\mathbb{E}}
\def\Nb{\mathbb{N}}
\def\Pb{\mathbb{P}}

\def\Rb{\mathbb{R}}

\def\Zb{\mathbb{Z}}
\def\Ac{\mathcal{A}}

\def\Hc{\mathcal{H}}
\def\Ic{\mathcal{I}}

\def\Lc{\mathcal{L}}

\def\Rc{\mathcal{R}}
\def\Sc{\mathcal{S}}
\def\Tc{\mathcal{T}}

\def\Vc{\mathcal{V}}

\def\Yc{\mathcal{Y}}

\def\Er{\mathbf{E}}
\def\Pr{\mathbf{P}}

\def\Ls{\mathscr{L}}
\def\Ts{\mathscr{T}}
\def\Et{\mathit{E}}
\def\Pt{\mathit{P}}

\def\th{^{\text{th}}}
\def\oT{{\overline{\Tc}}}
\def\oH{{\overline{\Hc}}}
\def\oR{{\overline{\rho}}}

\def\|{\\ \smallskip}
\def\d{\mathrm{d}}

\def\ind{\mathbf{1}}
\def\nin{n \rightarrow \infty}

\def\limn{\lim_{\nin}}

\def\ed{\stackrel{\text{\tiny{d}}}{=}}

\def\qqquad{\qquad \quad}

\newtheorem{thm}{Theorem}
\newtheorem{lem}{Lemma}[section]
\newtheorem{prp}[lem]{Proposition}
\newtheorem{cly}[lem]{Corollary}

\numberwithin{equation}{section} 
\title{Non-Gaussian fluctuations of randomly trapped random walks}
\author{Adam Bowditch, National University of Singapore}
\date{}    
\begin{document}
\maketitle 
\begin{abstract}
In this paper we consider the one-dimensional, biased, randomly trapped random walk when the trapping times have infinite variance. 
We prove sufficient conditions for the suitably scaled walk to converge to a transformation of a stable L\'{e}vy process. 
As our main motivation, we apply subsequential versions of our results to biased walks on subcritical Galton-Watson trees conditioned to survive.
This confirms the correct order of the fluctuations of the walk around its speed for values of the bias that yield a non-Gaussian regime. 
\end{abstract}

\let\thefootnote\relax\footnote{\textit{MSC2010 subject classifications:} 60K37, 60F17, 60G51, 60J80. }

\section{Introduction}\label{s:int}
Randomly trapped random walks (RTRWs) were introduced in \cite{arcacero15} to generalise models such as the Bouchaud trap model (see \cite{bo92, foma14, mo11, zi09}), as well as to provide a framework for studying random walks on other random graphs in which trapping naturally occurs such as random walks on percolation clusters (see \cite{dh84, frha14}) and random walk in random environment (see \cite{bosz02, kekosp75, szze99}). In the inaugural paper, the authors prove that the possible scaling limits of the unbiased RTRW on $\Zb$ belong to a class of time changed Brownian motions (called randomly trapped Brownian motions) including the fractional kinetics process (see \cite{mesc04}) and the FIN diffusion (see \cite{foisne02}) as well as a larger class where the time change retains much of the random spatial inhomogeneity in a more complex manner than in the FIN case. Unbiased RTRWs on $\Zb^d$ for $d\geq 2$ have been studied further in \cite{cewa15} where a complete classification of the possible scaling limits is given; generalising previously well known results for models such as the continuous time random walk \cite{mesc04} and the Bouchaud trap model \cite{arce07}.

In this paper, we are interested in biased RTRWs in one dimension motivated by the study of biased random walks on subcritical Galton-Watson (GW) trees conditioned to survive. Specifically, we study regimes in which the trapping is sufficiently strong so that the walk experiences non-Gaussian fluctuations. In recent years there has been much progress in models involving trapping; reviews of recent developments, more extensive background and further motivation is given in \cite{arce06} and \cite{arfr16}.  

We now define the one dimensional RTRW via a random time change of a simple random walk. Set a bias parameter $\beta\geq1$ and a law $\pi$ on $(0,\infty)$-valued probability measures. Fix an environment $\omega=(\omega_x)_{x \in \Zb}$ distributed according to the environment law $\Pr:=\pi^{\otimes \Zb}$. Let $(Y_k)_{k\geq0}$ be a simple random walk on $\Zb$ starting from $Y_0=0$ with transition probabilities $\Pt^\omega(Y_{k+1}=y-1|Y_k=y)=(\beta+1)^{-1}=1-\Pt^\omega(Y_{k+1}=y+1|Y_k=y)$. For $x \in \Zb$ write 
\[\Lc(x,n):=\sum_{k=0}^n\ind_{\{Y_k=x\}}\]
for the local time of $Y$ at site $x$ by time $n$. Under the quenched law $\Pt^\omega$, let $(\eta_{x,i})_{x \in \Zb, i\geq 1}$ be independent with $\eta_{x,i}$ distributed according to the law $\omega_x$. We define the clock process
\begin{flalign*}
 S_n \; := \; \sum_{x \in \Zb} \sum_{i=1}^{\Lc(x,n-1)}\eta_{x,i} \; = \; \sum_{k=0}^{n-1}\eta_{Y_k,\Lc(Y_k,k)} \qquad \text{ and write } \qquad S^{-1}_t:=\sup\{k\geq 0:S_k\leq t\}
\end{flalign*}
for its inverse. We then define the randomly trapped random walk by
\[X_t:=Y_{S^{-1}_t}.\]

For a fixed environment $\omega$, this process is then a continuous time random walk on $\Zb$ with $k\th$ holding time $\eta_k:=\eta_{Y_k,\Lc(Y_k,k)}$ distributed according to the $k\th$ trap $\omega_{Y_k}$ seen by the embedded walk. We note that the law of $Y$ is independent of the environment and we often drop the superscript from $\Pt^\omega$ when considering the embedded walk to emphasise this. For convenience we will define $S_t=S_{\lfloor t\rfloor}$ where $\lfloor t\rfloor:=\max\{k \in \Zb:k\leq t\}$ for non-integer $t \in \Rb$. The main results in this article will be with respect to the annealed law $\Pb(\cdot):=\int \Pt^\omega(\cdot) \Pr(\d \omega)$. 

Although we have the same underlying graph as studied in \cite{arcacero15}, we do not observe time changes which retain the randomness of the spatial composition in the same way. The reason for this is that the bias constantly forces the walk into new regions of the graph making it unlikely that the walk spends a large amount of time in any finite region. In particular, by a law of large numbers it is straightforward to see that $Y_{nt}/n$ converges $\Pb$-a.s.\ to the deterministic process $\upsilon_\beta t$ where 
\begin{flalign}\label{e:ups}
\upsilon_\beta=\frac{\beta-1}{\beta+1}.
\end{flalign}

A law of large numbers and central limit theorems for the randomly trapped random walk on $\Zb$ have been shown in \cite{bo18_1}. That is, it is shown that if the bias is positive ($\beta>1$) and the expected holding time $\Eb[\eta_0]$ is finite, then $X_{nt}/n$ converges $\Pb$-a.s.\ to the deterministic process $\nu_\beta t$ where 
\begin{flalign}\label{e:spd}
\nu_\beta=\upsilon_\beta\Eb[\eta_0]^{-1} 
\end{flalign}
is called the speed of the walk. Following this, an annealed invariance principle is shown which states that, if $\beta>1$ and $\Eb[\eta_0^2]<\infty$ then for some $\varsigma \in (0,\infty)$
\[B_t^n:=\frac{X_{nt}-nt\nu_\beta}{\varsigma \sqrt{n}}\]
converges in $\Pb$-distribution to a Brownian motion. 

In this article we consider the randomly trapped random walk when the trapping is too strong to obtain a central limit theorem. In this setting there are two distinct phases depending on whether $\Eb[\eta_0]<\infty$ or not. In both cases we prove functional limiting results for the position of the walk.

In Section \ref{s:sbl} we consider the sub-ballistic phase. That is, when the holding times have infinite mean, in Theorem \ref{t:Sstab} we give conditions under which the clock process $S_t$ can be rescaled to converge to a L\'{e}vy process (see \cite{be96} for background on L\'{e}vy processes). We then use this to show that the position of the walk converges, after suitable rescaling, to the scaled inverse of the aforementioned process. 

By $a\asymp b$ we mean that there exist constants $c$, $C$ such that $cb\leq a\leq Cb$ and by $a\sim b$ we mean that $a/b$ converges to a positive constant. Throughout, we use $D_U$, $D_{J_1}$ and $D_{M_1}$ to denote the space of c\`{a}dl\`{a}g functions mapping $[0,\infty)$ to $\Rb$ equipped with the uniform, $J_1$ and $M_1$ topologies respectively; we refer the reader to \cite{bi99,wh02} for a detailed set up of these spaces.   
\begin{thm}\label{t:Sstab}
Let $\beta>1$, $\alpha\in(0,1)$ and $f:\Nb\times\Rb_+\rightarrow \Rb$ be such that $f(1,\lambda)\asymp\lambda^\alpha$. Suppose there exists $n_k\nearrow \infty$, $a_n=n^{1/\alpha}L(n)$ for some $L$ slowly varying such that for any $l \in\Nb$ and $\lambda>0$ we have
\begin{flalign}\label{StabCon}
-n_k\log\left(\Er\left[\Et^\omega\left[\exp\left(-\frac{\lambda}{a_{n_k}}\eta_0\right)\right]^l\right]\right)\rightarrow f(l,\lambda)
\end{flalign}
as $k\rightarrow \infty$. Then, as $k\rightarrow \infty$, $(S_{n_kt}/a_{n_k},X_{a_{n_k}t}/n_k)_{t\geq 0}$ converges in $\Pb$-distribution on $D_{M_1}\times D_U$ to $(\Sc_t,\upsilon_\beta\Sc^{-1}_t)_{t\geq 0}$ where $\Sc_t$ is a L\'{e}vy process with Laplace transform
\[\Eb\left[\exp\left(-\lambda\Sc_t\right)\right]=\exp\left(-t\upsilon_\beta^2\sum_{l=1}^\infty f(l,\lambda)(1-\upsilon_\beta)^{l-1}\right).\]
In particular, if \eqref{StabCon} holds for all subsequences $n_k$ then $f(l,\lambda)=g(l)\lambda^\alpha$ for some $g\not\equiv0$ and $\Sc_t$ is an $\alpha$-stable subordinator. 
\end{thm}

In Section \ref{s:sdf} we consider the case where the holding times have finite mean but infinite variance. In this setting we are able to apply the speed result from \cite{bo18_1} but the fluctuations are too large to obtain a central limit theorem. Here, we prove conditions under which, after suitable centring and rescaling, the time taken to reach the $nt\th$ level of the backbone converges to a L\'{e}vy process. We then use results of \cite{wh02} to show that the centred and rescaled walk converges in distribution to a L\'{e}vy process which we describe in greater detail in Section \ref{s:sdf} following the proof. Write $\Delta_m:=\inf\{t> 0: X_t=m\}$ for the first hitting time of $m$. The main result of Section \ref{s:sdf} is Theorem \ref{t:fluc}.
\begin{thm}\label{t:fluc}
Let $\beta>1$, $\alpha\in(1,2)$ and $f:\Nb\times\Rb_+\rightarrow \Rb$ be such that $f(1,\lambda)\asymp\lambda^\alpha$. Suppose there exists $n_k\nearrow \infty$, $a_n=n^{1/\alpha}L(n)$ for some $L$ slowly varying such that for any $l \in\Nb$ and $\lambda>0$ we have
\begin{flalign}\label{LogAss}
n_k\log\left(\Er\left[\Et^\omega\left[\exp\left(-\frac{\lambda}{a_{n_k}}(\eta_0-\Eb[\eta_0])\right)\right]^l\right]\right)\rightarrow  f(l,\lambda)
\end{flalign}
as $k\rightarrow \infty$. Then, as $k\rightarrow \infty$, 
\[\left(\frac{\Delta_{n_kt\nu_\beta}-n_kt}{a_{n_k}}, \frac{X_{n_kt}-n_kt\nu_\beta}{a_{n_k}}\right)_{t\geq 0}\]
converges in $\Pb$-distribution on $D_{M_1}\times D_{M_1}$ to $(\Vc_t, -\nu_{\beta}\Vc_t)_{t\geq 0}$ where $\Vc_t$ is a L\'{e}vy process.
In particular, if \eqref{LogAss} holds for all subsequences $n_k$ then $f(l,\lambda)=g(l)\lambda^\alpha$ for some $g\not\equiv0$ and $\Vc_t$ is an $\alpha$-stable process. 
\end{thm}

The two asymptotic conditions (\ref{StabCon}) and (\ref{LogAss}) appear to be quite technical however they relate to the usual stable law conditions. Specifically, $\eta_0$ belongs to the domain of attraction of a stable law of index $\alpha$ with respect to $\Pb$ if and only if the relevant asymptotic holds with $l=1$ for all subsequences $n_k$ and $f(1,\lambda)=c\lambda^\alpha$ for some constant $c$ (e.g.\ \cite[Chapter 4]{wh02}). We need to extend to $l\geq 2$ to account for the walk visiting vertices multiple times and we consider more general $f$ so that we may consider holding times which do not belong to the domain of attraction of any stable law.  

It is worth noting that three of the four processes converge in the $M_1$ topology; it is natural to consider whether convergence can be extended to the $J_1$ topology. The main issue we encounter is that if the slowing is caused by spending a large amount of time in a few large traps in the environment then the total time spent in one of these traps may consist of several long excursions. This results in several large jumps in the discrete process $S_{nt}$ contributing to a single large jump in the limit. This means that the convergence will not hold under the $J_1$ topology which distinguishes between a series of small jumps in quick succession and a single large jump of the combined magnitude. A similar argument can be used for the position of the walk and in Section \ref{s:non} we use the example of random walk in random scenery to show that, in general, $(X_{nt}-nt\nu_\beta)/a_n$ does not converge in $D_{J_1}$. 

In Section \ref{s:GWTree} we consider our main motivating example of random walks on subcritical GW-trees. A subcritical GW-tree conditioned to survive consists of a semi-infinite path (called the backbone) with finite GW-trees as leaves. These finite branches form traps in the environment which slow the walk. Typically, the branches are quite short therefore the walk on the tree does not deviate far from the backbone. For this reason we have that the walk on the tree behaves like a randomly trapped random walk on $\Zb$ with holding times distributed as excursion times in finite GW-trees. Biased walks on subcritical GW-trees are, therefore, a natural example of the randomly trapped random walk. Furthermore, they exhibit interesting behaviour as the relationship between the bias and offspring law influences the trapping. In Theorem \ref{t:GWT} we determine the correct range of biases so that the walk undergoes non-Gaussian fluctuations in the spirit of Theorem \ref{t:fluc}. To avoid introducing a large amount of notation here, we defer an explicit construction of the model and an overview the known results to Section \ref{s:GWTree}.
%Due to a lattice effect (e.g.\ \cite{arfrgaha12, bo18}) the holding times in the GW-tree model do not satisfy the asymptotics in \eqref{StabCon} and \eqref{LogAss} which means that Theorems \ref{t:Sstab} and \ref{t:fluc} cannot be applied directly to this model. In Section \ref{s:GWTree} we discuss this further and prove a subsequential limiting result for this application in the essence of Theorem \ref{t:fluc}.

A key incentive for studying random walks on subcritical GW-trees is to understand random walks on supercritical GW-trees and percolation clusters. Although these models do not fit directly into the randomly trapped random walk framework, many of the estimates on excursion times and techniques developed here can be extended to (or used directly in) these models. Analogues of our results in these models would confirm several conjectures from \cite{arfr16} which we discuss further in Section \ref{s:GWTree}. 

A frequent question for random walk in random environment models is whether annealed convergence results can be extended to the quenched setting. This typically relies on properties of the embedded walk and the homogeneity of the random environment. For example, using a technique developed in \cite{bosz02}, it can be shown (see \cite{cewa15}) that, for a randomly trapped random walk that satisfies an annealed invariance principle in sufficiently high dimension, the convergence extends to a quenched invariance principle. This typically fails in low dimension where the local time of the walk is concentrated on a smaller collection of vertices. Indeed, it has been shown in \cite{bo18_1} that a biased randomly trapped random walk in one dimension satisfying an annealed invariance principle requires an environment dependent correction in order that convergence occurs in the quenched setting. This behaviour extends to the regime we consider here and therefore extending to the quenched setting is substantial.

\section{Sub-ballistic regime}\label{s:sbl}
In this section we classify limits of the RTRW model where the embedded random walk has a positive bias and the holding times have infinite mean. The main aim is to prove Theorem \ref{t:Sstab} which we do by considering the arguments of \cite{cewa15} and applying results of \cite{wh02}. For $x \in \Zb$ write 
\[\hat{\omega}_x(\lambda)=\int e^{-\lambda u}\omega_x(\d u)=\Et^\omega[e^{-\lambda \eta_{x,1}}]\] 
to be the quenched Laplace transform of $\omega_x$. Recall from \eqref{e:ups} that $\upsilon_\beta=(\beta-1)/(\beta+1)$ is the speed of the embedded walk $Y$. By the Gambler's ruin this is the probability that $Y$ never returns to the origin. 
 For convenience we write $\Pb^Y(\cdot):=\Pb(\cdot|Y)$ and $\Pt^{\omega,Y}(\cdot):=\Pt^\omega(\cdot|Y)$. We begin, in Proposition \ref{p:Sstab}, by showing that the one-dimensional distributions of the scaled clock process converge.

\begin{prp}\label{p:Sstab}
Under the assumptions of Theorem \ref{t:Sstab}, for any $\lambda>0$,
\[\lim_{k\rightarrow \infty}\Eb\left[\exp\left(-\lambda \frac{S_{n_kt}}{a_{n_k}}\right)\right]=\exp\left(-\upsilon_\beta^2 t\sum_{l=1}^\infty f(l,\lambda)(1-\upsilon_\beta)^{l-1}\right).\]
  \begin{proof}
For simplicity we write $n$ for $n_k$ with the understanding that we consider subsequences along which \eqref{StabCon} holds. Conditional on $Y$, the holding times at different vertices are independent therefore
\begin{flalign*}
 \Eb^Y\left[\exp\left(-\lambda \frac{S_{nt}}{a_n}\right)\right] 
& = \;  \Eb^Y\left[\exp\left(-\frac{\lambda}{a_n}\sum_{x \in \Zb}\sum_{i=1}^{\Lc(x,\lfloor nt\rfloor-1)}\eta_{x,i}\right)\right] \\
 & = \; \prod_{x \in \Zb} \Eb^Y\left[\exp\left(-\frac{\lambda}{a_n}\sum_{i=1}^{\Lc(x,\lfloor nt\rfloor-1)}\eta_{x,i}\right)\right]. %\\
 \end{flalign*}
 Under $\Pt^{\omega,Y}$ the holding times at a given vertex are independent and do not depend on $Y$ therefore we can write the above expression as
 \begin{flalign*}
  \prod_{x \in \Zb} \Er^Y\left[\Et^{\omega,Y}\left[\prod_{i=1}^{\Lc(x,\lfloor nt\rfloor-1)}e^{-\frac{\lambda}{a_n}\eta_{x,i}}\right]\right] 
 & = \prod_{x \in \Zb}\Er^Y\left[\hat{\omega}_x\left(\frac{\lambda}{a_n}\right)^{\Lc(x,\lfloor nt\rfloor -1)}\right] \\
 & = \prod_{l=1}^\infty\prod_{x \in \Rc_l(nt)}\Er\left[\hat{\omega}_x\left(\frac{\lambda}{a_n}\right)^l\right]
\end{flalign*}
where $\Rc_l(m)=\{x \in \Zb: \; \Lc(x,m-1)=l\}$ is the number of vertices with local time $l$ at time $m-1$. By translation invariance we then have that
\begin{flalign*}
 \Eb^Y\!\left[\exp\left(-\lambda \frac{S_{nt}}{a_n}\right)\right]
 &\; = \; \prod_{l=1}^\infty\Er\!\left[\hat{\omega}_0\left(\frac{\lambda}{a_n}\right)^l\right]^{|\Rc_l(nt)|} \\
&\! =\; \exp\left(\sum_{l=1}^\infty |\Rc_l(nt)|\log\left(\Er\!\left[\hat{\omega}_0\left(\frac{\lambda}{a_n}\right)^l\right]\right)\right).
\end{flalign*}

Write $\Rc_l^+(m)=\{x \in \Zb: \; \Lc(x,m-1)\geq l\}$ then by \cite{pi74} we have that, for fixed $t>0$, $|\Rc_l(nt)|/\lfloor nt\rfloor$ converges almost surely to $\upsilon_\beta^2(1-\upsilon_\beta)^{l-1}$ and $|\Rc_l^+(nt)|/\lfloor nt\rfloor$ converges almost surely to $\upsilon_\beta(1-\upsilon_\beta)^{l-1}$ as $\nin$. The local time at the origin stochastically dominates that of any other vertex therefore
\begin{flalign*}
 \frac{\Et[|\Rc_l^+(m)|]}{m} \; = \; \frac{1}{m}\sum_{x=-m}^m \Pt(\Lc(x,m)\geq l) \;  \leq \; \frac{2m +1}{m}\Pt(\Lc(0,m)\geq l).
\end{flalign*}
The local time $\Lc(0,m)$ is increasing in $m$, therefore for each $l\geq 1$ we have that $\Pt(\Lc(0,m)\geq l)\leq\Pt(\Lc(0,\infty)\geq l)= (1-\upsilon_\beta)^{l-1}$ since $\Lc(0,\infty)$ is geometrically distributed with termination probability $\upsilon_\beta$. It follows that 
\begin{flalign}\label{Rbarbnd}
\frac{\Et[|\Rc_l^+(nt)|]}{nt} \leq C(1-\upsilon_\beta)^l  
\end{flalign}
uniformly over $n,t,l$. For fixed $M \in \Nb$, by using the convergence of $|\Rc_l(nt)|/\lfloor nt\rfloor$ and the assumptions on $\hat{\omega}$ we have that for $\Pt$-a.e.\ $Y$
\begin{flalign*}
 \sum_{l=1}^M |\Rc_l(nt)|\log\left(\Er\left[\hat{\omega}_0\left(\frac{\lambda}{a_n}\right)^l\right]\right) \rightarrow -\upsilon_\beta^2t\sum_{l=1}^Mf(l,\lambda)(1-\upsilon_\beta)^{l-1}. 
\end{flalign*}

By Jensen's inequality and (\ref{StabCon}), 
\[-\log(\Er[\hat{\omega}_0(\lambda/a_n)^l])\;\leq\; -l\log(\Er[\hat{\omega}_0(\lambda/a_n)]) \;\leq\; Cl\lambda^\alpha n^{-1}\]
thus $f$ grows at most linearly in $l$ and therefore we have that $\sum_{l=1}^Mf(l)(1-\upsilon_\beta)^{l-1}$ converges as $M\rightarrow \infty$. Moreover,
\begin{flalign*}
 -\sum_{l=M}^\infty |\Rc_l(nt)|\log\left(\Er\left[\hat{\omega}_0\left(\frac{\lambda}{a_n}\right)^l\right]\right)
 \;\leq\; \frac{C\lambda^\alpha}{n}\sum_{l=M}^\infty l|\Rc_l(nt)|.
\end{flalign*}
By Markov's inequality and (\ref{Rbarbnd}) we have that for $c_1 \in (0,-\log(1-\upsilon_\beta))$ and some $\tilde{c}>0$
\[\Pt\left(\frac{|\Rc_l^+(nt)|}{nt}\geq e^{-c_1l}\right)\;\leq\; \Et\left[\frac{|\Rc_l^+(nt)|}{nt}\right]e^{c_1l} \;\leq\; C(1-\upsilon_\beta)^l e^{c_1l} \;\leq\; Ce^{-\tilde{c}l}.\]
A union bound then shows that 
\[\Pt\left(\bigcup_{l\geq M} \left\{\frac{|\Rc_l(nt)|}{nt}\geq e^{-c_1l}\right\}\right) \leq \sum_{l\geq M}e^{-\tilde{c}l} \leq Ce^{-\tilde{c}M}.\]
It then follows that, for $\delta>0$, 
\begin{flalign*}
 \Pt\left( -\sum_{l=M}^\infty |\Rc_l(nt)|\log\left(\Er\left[\hat{\omega}_0\left(\frac{\lambda}{a_n}\right)^l\right]\right)\geq \delta\right) 
 \;\leq\; Ce^{-\tilde{c}M}+\ind_{\left\{Ct\lambda^\alpha\sum_{k=M}^\infty le^{-\tilde{c}l}\geq \delta\right\}}
\end{flalign*}
% \begin{flalign*}
%  & \Pt\left(-\log\left(\Er\left[\hat{\omega}_0\left(\frac{\lambda}{a_n}\right)\right]\right)\sum_{k=M}^\infty k|\Rc_l(nt)|\geq \delta\right) \\
% % & \qqqquad  \leq Ce^{-\tilde{c}M}+\Pb\left(-\frac{\log\left(\Er\left[\hat{\omega}_0\left(\frac{\lambda}{a_n}\right)\right]\right)}{nt}\sum_{k=M}^\infty ke^{-\tilde{c}k}\geq \delta\right)
% & \qqqquad  \leq Ce^{-\tilde{c}M}+\ind_{\left\{-\log\left(\Er\left[\hat{\omega}_0\left(\frac{\lambda}{a_n}\right)\right]\right)\sum_{k=M}^\infty ke^{-\tilde{c}k}\geq nt\delta\right\}}
% \end{flalign*}
which converges to $0$ as $M\rightarrow \infty$ independently of $n$. In particular, 
\begin{flalign*}
 \sum_{l=1}^\infty |\Rc_l(nt)|\log\left(\Er\left[\hat{\omega}_0\left(\frac{\lambda}{a_n}\right)^l\right]\right) \rightarrow -\upsilon_\beta^2 t\sum_{l=1}^\infty f(l,\lambda)(1-\upsilon_\beta)^{l-1} 
\end{flalign*}
in $\Pt$-probability therefore, by bounded convergence, we have that 
\[\limn\Eb\left[\exp\left(-\lambda \frac{S_{nt}}{a_n}\right)\right]=\exp\left(-\upsilon_\beta^2 t\sum_{l=1}^\infty f(l,\lambda)(1-\upsilon_\beta)^{l-1}\right).\]
  \end{proof}
\end{prp}

Before proving the main result we state a technical lemma which allows us to compare the process with a version in which we re-sample the environment at fixed times. This will allow us to show that the limiting process has independent increments. The result follows similarly to \cite[Lemma 3.2]{cewa15} therefore we omit the proof.

Fix $d\in \Nb$, $0<t_1<...<t_d=t$ and let $\overline{\omega}:=(\omega_x^j)_{x \in \Zb, j=1,...,d}$ be a sequence of i.i.d.\ $\pi$ random measures. For $x \in \Zb$, $j=1,...,d$ and $i\geq 1$ let $\eta_{x,i}^j$ be independent with law $\omega_x^j$. Let $j(x)$ be such that $nt_{j(x)-1}\leq \Delta^Y_x< nt_{j(x)}$ denote the index of the interval in which $x$ is first reached. Define 
\[S'_m:=\sum_{x \in \Zb}\sum_{i=1}^{\Lc(x,m-1)}\eta_{x,i}^{j(x)}\]
then $S'\ed S$ with respect to the annealed law $\Pb$. The clock process $S'_m$ can be thought of as the sum of the first $m$ holding times where we refresh the entire unseen environment at times $nt_j$. We then define the approximation of $S'$: 
\[S''_m=\sum_{j=1}^d\sum_{k=nt_{j-1}}^{(m\land nt_j)-1} \eta_{Y_k,\Lc(Y_k,k)}^j.\]
The process $S''_m$ can be thought of as the sum of the first $m$ holding times where we refresh the entire environment at times $nt_j$. By independence of the environment and the walk between times $nt_{j-1}$ and $nt_j$ we have that the differences $S''_{nt_j}-S''_{nt_{j-1}}$ are independent.
\begin{lem}\label{l:indinc}
Under the assumptions of Theorem \ref{t:Sstab} we have that for any $\varepsilon>0$,
\[\Pb\left(\sup_{s \leq t}|S''_{ns}-S'_{ns}|>\varepsilon a_n\right)\]
converges to $0$ as $\nin$.
\end{lem}

Using these two results we can conclude the proof of Theorem \ref{t:Sstab} by applying results of \cite{wh02}.
 \begin{proof}[Proof of Theorem \ref{t:Sstab}]
We begin by completing the result for the clock process $S_t$. By Proposition \ref{p:Sstab} we have that the one-dimensional distributions of $S_{n_kt}/a_{n_k}$ converge and the limit $\Sc_t$ is $\Pb$-a.s.\ finite for any $t>0$. Moreover, $\Pb$-a.s.\ we have that $\Sc_0=0$ since 
\begin{flalign}\label{e:tZero}
\lim_{t \rightarrow 0}\Eb[\exp(-\lambda\Sc_t)] =1.
\end{flalign}
It then follows from \textit{equivalent characterisations of convergence for monotone functions} \cite[Theorem 13.6.3]{wh02} and that $S_{n_kt}$ in increasing in $t$ that $S_{n_kt}/a_{n_k}$ converges in distribution to the process $\Sc_t$ with respect to the $M_1$ topology. 

The clock process $S_m$ has stationary increments because the traps in the environment are i.i.d.\ therefore $\Sc$ also has stationary increments. By Lemma \ref{l:indinc} we have that $\Sc$ has independent increments. To show that $\Sc$ is a L\'{e}vy process it remains to show that for every $\varepsilon>0$ and $t\geq 0$ we have that $\lim_{h\rightarrow 0^+}\Pb(|\Sc_{t+h}-\Sc_t|>\varepsilon)=0$. This follows from \eqref{e:tZero} since $\Sc$ has stationary and independent increments.
 
The limit process $\Sc_t$ belongs to $D([0,\infty),\Rb)$, is unbounded and strictly increasing.
By \textit{continuity of the inverse at strictly increasing functions} \cite[Corollary 13.6.4]{wh02}, the inverse map between unbounded, increasing functions in $D_{M_1}$ onto $D_{U}$ is continuous at unbounded, strictly increasing functions. 
Therefore, the sequence $S^{-1}_{a_{n_k}t}/n_k$ converges in $\Pb$-distribution on $D_{U}$ to $\Sc^{-1}_t$. 
 
The main result then follows from \textit{continuity of composition at continuous limits} \cite[Theorem 13.2.1]{wh02} since $Y_{nt}/n$ converges $\Pb$-a.s.\ on $D_{U}([0,\infty),\Rb)$ to the process $t\upsilon_\beta$ and $X_{a_{n_k}t}/n_k=Y_{S^{-1}_{a_{n_k}t}}/n_k$.
 
We want to show that if the convergence in \eqref{StabCon} holds along all subsequences then $f(l,\lambda)=g(l)\lambda^\alpha$ for a function $g\not\equiv0$. Notice that
\begin{flalign}\label{e:fForm}-n\log\left(\Er\left[\Et^\omega\left[\exp\left(-\frac{\lambda}{a_{n}}\eta_0\right)\right]^l\right]\right)=-n\log\left(\Er\left[\Et^\omega\left[\exp\left(-\frac{1}{(n\lambda^{-\alpha})^{1/\alpha}L(n\lambda^{-\alpha})}\frac{L(n\lambda^{-\alpha})}{L(n)}\eta_0\right)\right]^l\right]\right).
\end{flalign}
Since $L$ is slowly varying, for $\lambda>0$ fixed we have that $L(n\lambda^{-\alpha})/L(n)$ converges to $1$ as $\nin$. In particular, using computations similar to Proposition \ref{p:Sstab}, it is straightforward to show that, for each $l\in\Nb$, the right hand side of \eqref{e:fForm} converges to $f(l,1)\lambda^\alpha$. By assumption we have that the left hand side of \eqref{e:fForm} converges to $f(l,\lambda)$ therefore we indeed have that $f(l,\lambda)=g(l)\lambda^\alpha$ for $g(l)=f(l,1)$.

To show that $\Sc_t$ is an $\alpha$-stable subordinator it suffices to note that $\Sc_t$ is increasing and self-similar:
\begin{flalign*}
\Sc_t \;=\; \limn \frac{S_{nt}}{a_n} \;=\; \limn \frac{a_{\lambda n}}{a_n}\cdot \frac{S_{\lambda n\frac{t}{\lambda}}}{a_{\lambda n}} \;\ed\; \Sc_{t/\lambda}\limn \frac{a_{\lambda n}}{a_n} \;=\; \lambda^{1/\alpha}\Sc_{t/\lambda}.
\end{flalign*}
\end{proof}

\section{Ballistic regime}\label{s:sdf}
In this section we consider the regime in which $\Eb[\eta_0]<\infty$ but $\Eb[\eta_0^2]=\infty$. In this setting we do not have a central limit theorem for $X$; however, by \cite[Corollary 2.2]{bo18_1} we have that if $\beta>1$ then $X_{nt}n^{-1}$ converges $\Pb$-a.s.\ to $\nu_\beta t$ where $\nu_\beta$ is given in \eqref{e:spd}. Recall that we define the first hitting times $\Delta_m:=\inf\{t\geq 0:X_t=m\}$. The main aim of the section is to prove Theorem \ref{t:fluc} which shows that $(\Delta_{nt\nu_\beta}-nt)/a_n$ and $(X_{nt}-nt\nu_\beta)/a_n$ converge jointly to a pair of L\'{e}vy processes under certain assumptions on the distribution of the holding times and for a suitable scaling $a_n$.

Our strategy will be as follows. We begin by using a regeneration structure in order to approximate $\Delta_{nt}$ by an i.i.d.\ sum of random variables. We then prove convergence of the Laplace transforms of these variables. This completes the results for $\Delta_{nt}$ which we then extend to the inverse by using standard inverse results for processes with linear centring. We then conclude the result for $X_{nt}$ by showing that it does not deviate too far from an appropriate transformation of the inverse of $\Delta_{nt}$.

%Write $\Delta_l^Y:=\inf\{m\geq 0:Y_m=l\}$ in order to study the process $X$.
We begin by exploiting the renewal structure of the walk. Let $\zeta^Y_0=0$ and, for $j=1,2,...$, define $\zeta^Y_j:=\inf\{m>\zeta^Y_{j-1}: \; \{Y_r\}_{r=0}^{m-1} \cap \{Y_r\}_{r=m}^\infty = \emptyset \}$ to be the regeneration times of the walk $Y$. We then have that $\zeta^X_j:=S_{\zeta^Y_j}$ for $j\geq 1$ are regeneration times for $X$, $\varrho_j:=X_{\zeta^X_j}=Y_{\zeta^Y_j}$ are regeneration points and we write 
\begin{flalign*}
\chi^+_j & \;:=\; \nu_\beta\left((\zeta_j^X-\zeta_{j-1}^X)-\Eb[\eta_0](\zeta^Y_j-\zeta^Y_{j-1})\right), \\
\chi^-_j & \;:=\; \upsilon_\beta(\zeta^Y_j-\zeta^Y_{j-1})-(\varrho_j-\varrho_{j-1}), \\
\chi_j&\;:=\; \nu_\beta\left(\zeta^X_j-\zeta_{j-1}^X\right)-(\varrho_j-\varrho_{j-1})\;=\;\chi^+_j+\chi^-_j.
\end{flalign*}
Similarly to \cite[Lemma 2.3]{bo18_1} and using a law of large numbers, we have that the families  $\{\chi_j^+\}_{j\geq 2}$, $\{\chi_j^-\}_{j\geq 2}$ and $\{\chi_j\}_{j\geq 2}$ are centred and i.i.d.\ under $\Pb$ whenever $\beta>1$ and $\Eb[\eta_0]<\infty$. In particular, 
\[\Sigma_{m}:= \sum_{j=2}^{m}\chi_j=\left(\zeta_{m}^X\nu_\beta-X_{\zeta_{m}^X}-\right) -\left(\zeta_1^X\nu_\beta-X_{\zeta_1^X}\right)\]
is a sum of i.i.d.\ centred random variables which we will later use to approximate $\Delta_{nt\nu_\beta}-nt$ along a certain time change. We first show that $\Sigma_{n_kt}/a_{n_k}$ converges to a L\'{e}vy process which we begin, in Lemma \ref{l:KillY}, by showing that $\{\chi_j^-\}_{j\geq 2}$ are negligible. 

\begin{lem}\label{l:KillY}
Suppose $\beta>1$ and $T<\infty$. If $a_n>\!>n^{1/2}$ then 
\[\sup_{t\leq T}\;\left|\sum_{j=2}^{nt}\frac{\chi_j^-}{a_n}\right|\]
converges in $\Pb$-probability to $0$ as $\nin$.
\begin{proof}
  Define 
  \[\Sigma_m^-:=\sum_{j=2}^m\chi_j^-\]
  which is a martingale with centred and independent jumps. The time and distance between regenerations of a biased random walk on $\Zb$ have exponential moments by \cite[Lemma 5.1]{depeze96} therefore $\chi_j^-$ also have exponential moments. 
  
By Doob's inequality we then have that for $\varepsilon>0$,
\begin{flalign*}
\Pb\left(\sup_{m\leq nT}|\Sigma_m^-|\geq \varepsilon a_n\right)  \;\leq\; \frac{\Eb[\Sigma_{nT}^-]}{\varepsilon^2a_n^2} \;\leq\; \frac{C_{T,\varepsilon}n}{a_n^2}\Eb\left[\left(\chi_2^-\right)^2\right]
\end{flalign*}
which converges to $0$ as $\nin$ since $n/a_n^2\rightarrow 0$ and $\chi_2^-$ has exponential moments.
\end{proof}
\end{lem}

We now turn to proving convergence of $\Sigma_{m}^+:=\sum_{j=2}^{m}\chi_j^+$ after suitable scaling which we do by considering the Laplace transform. We first prove a technical lemma that will allow us to use dominated convergence in the proof of this result.
\begin{lem}\label{l:UpLw}
  Under the assumptions of Theorem \ref{t:fluc} there exist constants $c$, $C$ depending only on $\lambda$ such that for sufficiently large $k$ (independently of $l$)
\[ cl \; \leq \; n_k\log\left(\Er\left[\Et^\omega\left[\exp\left(-\frac{\lambda}{a_{n_k}}(\eta_0-\Eb[\eta_0])\right)\right]^l\right]\right) \; \leq \; Cl^2.\]
\begin{proof}
For simplicity we write $n$ for $n_k$ with the understanding that we consider subsequences along which \eqref{LogAss} holds.
By Jensen's inequality, for a positive random variable $Z$ and $l \in \Zb$ satisfying $\Eb[Z^l]<\infty$ we have
%\begin{flalign}\label{e:Zen}
%\Eb[Z]^l  \; \leq \; \Er\left[\Et^\omega\left[Z\right]^l\right]\; \leq \;  \Eb[Z^l].
%\end{flalign}
%It therefore follows that
\begin{flalign}\label{LogBnd}
l\log\left(\Er\left[\Et^\omega\left[Z\right]\right]\right)\;\leq \; \log\left(\Er\left[\Et^\omega\left[Z\right]^l\right]\right) \; \leq \; \log\left(\Er\left[\Et^\omega\left[Z^l\right]\right]\right).
\end{flalign}
Choose 
\[Z(\lambda)=\exp\left(-\frac{\lambda}{a_n}(\eta_0-\Eb[\eta_0])\right) \quad\mathrm{then}\quad  Z(\lambda)^l=\exp\left(-\frac{l\lambda}{a_n}(\eta_0-\Eb[\eta_0])\right)=Z(l\lambda)\]
however, by (\ref{LogAss}),
\[n\log\left(\Er\left[\Et^\omega\left[Z\right]^l\right]\right)\sim f(l,\lambda) \qquad \mathrm{and} \qquad n\log\left(\Er\left[\Et^\omega\left[Z^l\right]\right]\right)\sim f(1,l\lambda)\leq C(l\lambda)^\alpha.\]
The bound (\ref{LogBnd}) then implies that $cl\lambda^\alpha\leq f(l,\lambda)\leq C(l\lambda)^\alpha$ for positive constants $c,C$. In particular, $f(l,\lambda)>0$ for all $\lambda>0$. 

The lower bound in the statement follows from (\ref{LogAss}) with $l=1$ and \eqref{LogBnd}:
%\begin{flalign*}
%&n\log\left(\Er\left[\Et^\omega\left[\exp\left(-\frac{\lambda}{a_n}(\eta_0-\Eb[\eta_0])\right)\right]^l\right]\right) \\
%&\qqqqquad \geq ln\log\left(\Er\left[\Et^\omega\left[\exp\left(-\frac{\lambda}{a_n}(\eta_0-\Eb[\eta_0])\right)\right]\right]\right) \\
%&\qqqqquad \geq c_\lambda l.
%\end{flalign*}
\begin{flalign*}
n\log\left(\Er\left[\Et^\omega\left[\exp\left(-\frac{\lambda}{a_n}(\eta_0-\Eb[\eta_0])\right)\right]^l\right]\right) 
 \geq ln\log\left(\Er\left[\Et^\omega\left[\exp\left(-\frac{\lambda}{a_n}(\eta_0-\Eb[\eta_0])\right)\right]\right]\right) 
 \geq c_\lambda l.
\end{flalign*}

For the upper bound, using \eqref{LogBnd} we have that
 \begin{flalign}
n\log\left(\Er\left[\Et^\omega\left[\exp\left(-\frac{\lambda}{a_n}(\eta_0-\Eb[\eta_0])\right)\right]^l\right]\right)%\notag\\
 \leq %n\log\left(\Eb\left[\exp\left(-\frac{\lambda l}{a_n}\left(\eta_0-\Eb[\eta_0]\right)\right)\right]\right) \notag \\ =
    \frac{n\lambda l\Eb[\eta_0]}{a_n}+n\log\left(\Eb\left[\exp\left(-\frac{\lambda l}{a_n}\eta_0\right)\right]\right). \label{e:Jen}
 \end{flalign}
 
Using that $e^{-y}\leq 1- y+y^2/2$ for $y\geq 0$, we have
\begin{flalign*}
&\Eb\left[\exp\left(-\frac{\lambda l}{a_n}\eta_0\right)\right] - 1+\frac{\lambda l}{a_n}\Eb[\eta_0] \\
&\qquad = \Eb\left[\left(\exp\left(-\frac{\lambda l}{a_n}\eta_0\right)-1+\frac{\lambda l}{a_n}\Eb[\eta_0]\right)\ind_{\{\eta_0>a_n\}}\right] 
+\Eb\left[\left(\exp\left(-\frac{\lambda l}{a_n}\eta_0\right)-1+\frac{\lambda l}{a_n}\Eb[\eta_0]\right)\ind_{\{\eta_0\leq a_n\}}\right] \\
& \qquad \leq \lambda l\Eb\left[\frac{\eta_0}{a_n}\ind_{\{\eta_0> a_n\}}\right] 
+\left(\lambda l\right)^2\Eb\left[\left(\frac{\eta_0}{a_n}\right)^2\ind_{\{\eta_0\leq a_n\}}\right].
\end{flalign*}
It then follows by properties of stable laws (cf.\ \cite[Chapter XVII.5]{fe71}) that the sequences
\[   n\Eb\left[ \frac{\eta_0}{a_n}\ind_{\{\eta_0\geq a_n\}}\right]\quad \text{and}\quad n\Eb\left[ \left(\frac{\eta_0}{a_n}\right)^2\ind_{\{\eta_0\leq a_n\}}\right]\]
are bounded above. We therefore have that, for some constant $C$, 
\begin{flalign}\label{e:cont}
\Eb\left[\exp\left(-\frac{\lambda l}{a_n}\eta_0\right)\right] \leq 1 - \frac{\lambda l \Eb[\eta_0]}{a_n} +\frac{Cl^2}{n}
\end{flalign}
uniformly over $l\geq 1$.

We want to show that we can choose $N_0$ sufficiently large such that the right-hand side of (\ref{e:cont}) is strictly larger than zero for any $k\geq 1$ and $n\geq N_0$. Since $\alpha\in(1,2)$ we have that $n/a_n^2\rightarrow 0$ as $\nin$ therefore fix $N_0$ large enough such that, for $n\geq N_0$,
\[\frac{n(\lambda\Eb[\eta_0])^2}{Ca_n^2}<\frac{1}{2}.\]
We can write
\[1 - \frac{\lambda l \Eb[\eta_0]}{a_n} +\frac{Cl^2}{n}=1-\frac{l}{a_n}\left(\lambda\Eb[\eta_0]-\frac{Cla_n}{n}\right)\]
where the term in the brackets is only positive if $l<\lambda\Eb[\eta_0]n/(Ca_n)$. This means that 
\[\frac{l}{a_n}\left(\lambda\Eb[\eta_0]-\frac{Cka_n}{n}\right) \;\leq\; \frac{\lambda\Eb[\eta_0]n}{Ca_n^2}\cdot\lambda\Eb[\eta_0] \;\leq\; \frac{1}{2}\]
for any $n\geq N_0$. This shows that the right-hand side of (\ref{e:cont}) is strictly larger than zero for any $l\geq 1$ and $n\geq N_0$.

Since $\log(1+x)\leq x$ for $x>-1$ we then have that 
\[n\log\left(\Eb\left[\exp\left(-\frac{\lambda l}{a_n}\eta_0\right)\right]\right)\leq - \frac{n\lambda l\Eb[\eta_0]}{a_n} +Cl^2. \]
Combining this with (\ref{e:Jen}) we then have that 
\[ n\log\left(\Er\left[\Et^\omega\left[\exp\left(-\frac{\lambda}{a_n}(\eta_0-\Eb[\eta_0])\right)\right]^l\right]\right)\leq Cl^2.\]
\end{proof}
\end{lem}

We now show that the scaled sum of the random variables $\chi_j^+$ converges along the given subsequences. Let $\Ic:=\{Y_j\}_{j=\zeta^Y_1}^{\zeta^Y_2-1}$ denote the set of vertices in the first regeneration block and for $l=1,2,...$ let $\Ic_l:=\{x \in \Ic: \; \Lc(x,\infty)=l\}$ denote those visited exactly $l$ times.   
\begin{prp}\label{p:DoA}
 Under the assumptions of Theorem \ref{t:fluc} 
 \[\lim_{k\rightarrow \infty}\Eb\left[\exp\left(-\frac{\lambda}{a_{n_k}}\sum_{j=2}^{tn_k}\chi^+_j\right)\right]=\exp\left(t\sum_{l=1}^\infty f(l,\lambda\nu_\beta)\Et[|\Ic_l|]\right).\]
\begin{proof}
As in Lemma \ref{l:UpLw} we write $n$ for $n_k$ with the understanding that we consider subsequences along which \eqref{LogAss} holds.
Since $\chi_j^+$ are i.i.d.\ we have that
\begin{flalign*}
\Eb\left[\exp\left(-\frac{\lambda}{a_n}\sum_{j=2}^{tn}\chi^+_j\right)\right] = \left(\frac{n\Eb\left[\exp\left(-\frac{\lambda}{a_n}\chi^+_2\right)-1\right]}{n}+1\right)^{tn-1}
\end{flalign*}
therefore it suffices instead to show that 
\begin{equation*}
\limn n\Eb\left[\exp\left(-\frac{\lambda}{a_n}\chi^+_2\right)-1\right]=\sum_{l=1}^\infty f(l,\lambda)\Et[|\Ic_l|].
\end{equation*}

We first show that $\Et\left[|\Ic_l|\right]$ decays exponentially in $l$. Since $\varrho_1\geq 1$ we have that
\[\Et\left[|\Ic_l|\right] \; \leq \; \Et\left[\sum_{x \in \Ic}\ind_{\{\Lc(x,\infty)\geq l\}}\right] \; = \; \sum_{x \in \Nb}\Et\left[\ind_{\{x \in \Ic\}}\ind_{\{\Lc(x,\infty)\geq l\}}\right]. \]
Using Cauchy-Schwarz and that $\Lc(x,\infty)$ are identically distributed for $x \in \Nb$ we have that this is bounded above by
\[ \sum_{x \in \Nb}\Pt\left(x \in \Ic\right)^{1/2}\Pt\left(\Lc(x,\infty)\geq l\right)^{1/2} \; = \;  \Pt\left(\Lc(0,\infty)\geq l\right)^{1/2} \sum_{x \in \Nb}\Pt\left(x \in \Ic\right)^{1/2}. \]
The number of returns to the origin is geometrically distributed with probability $\upsilon_\beta$ of moving away from the origin and never returning; therefore, $\Pt\left(\Lc(0,\infty)\geq l\right)= (1-\upsilon_\beta)^{l-1}$. By \cite[Lemma 5.1]{depeze96} we have that $\Pt\left(x \in \Ic\right)\leq \Pt(x \leq \zeta^Y_2) \leq Ce^{-cx}$. It follows that
\begin{flalign}\label{EI}
\Et\left[|\Ic_l|\right] \; \leq \; C\left(1-\upsilon_\beta\right)^{l/2}\sum_{x\geq 0}e^{-cx/2} \; \leq \; C\left(\frac{2}{\beta+1}\right)^{l/2}.
\end{flalign}

Conditional on $Y$, the holding times at different vertices are independent therefore
\begin{flalign*}
\Eb\left[\exp\left(-\frac{\lambda }{\nu_\beta a_n}\chi_2^+\right)\right]
&=\Eb\left[\exp\left(-\frac{\lambda}{a_n}\big(\zeta_2^X-\zeta_1^X-\Eb[\eta_0](\zeta^Y_2-\zeta^Y_1)\big)\right)\right] \\
&=\Eb\left[\exp\left(-\frac{\lambda}{a_n}\left(\sum_{x\in\Ic}\sum_{j=1}^{\Lc(x,\infty)}(\eta_{x,j}-\Eb[\eta_{x,j}])\right)\right)\right]\\
&=\Et\left[\prod_{x\in\Ic}\Eb^Y\left[\exp\left(-\frac{\lambda}{a_n}\left(\sum_{j=1}^{\Lc(x,\infty)}(\eta_{x,j}-\Eb[\eta_{x,j}])\right)\right)\right]\right].
  \end{flalign*}
Under $\Pt^{\omega,Y}$ the holding times at a given vertex are independent and identically distributed therefore we can write the above expression as
 \begin{flalign*}
  &\Et\left[\prod_{x \in \Ic}  \Er^Y\left[  \Et^{\omega,Y}\left[\prod_{j=1}^{\Lc(x,\infty)}  e^{-\frac{\lambda}{a_n}\left(\eta_{x,j}-\Eb[\eta_{x,j}]\right)}\right]\right]\right]
  \; = \; \Et\left[\prod_{k=1}^\infty\prod_{x \in \Ic_l}  \Er^Y\left[  \Et^{\omega,Y}\left[e^{-\frac{\lambda}{a_n}\left(\eta_{x,1}-\Eb[\eta_{x,1}]\right)}\right]^l\right]\right].
 \end{flalign*}
For $x \in \Zb$ let $\tilde{\eta}_x=\eta_{x,1}$ be the first holding time at $x$. Notice that, since $\tilde{\eta}_x$ is independent of $Y$, we have that
\[\Psi_{l,x}:=\Er^Y\left[  \Et^{\omega,Y}\left[e^{-\frac{\lambda}{a_n}\left(\eta_{x,1}-\Eb[\eta_{x,1}]\right)}\right]^l\right]=\Er\left[  \Et^{\omega}\left[e^{-\frac{\lambda}{a_n}\left(\tilde{\eta}_x-\Eb[\tilde{\eta}_x]\right)}\right]^l\right]\]
 is independent of $Y$ and, therefore, deterministic. In particular, since $\omega_x$ are i.i.d.\ under $\Pr$ we have that $\Psi_{l,x}$ does not depend on $x$ and we can write
\begin{flalign}\label{e:eLogPsi}
 \Eb\left[\exp\left(-\frac{\lambda}{ \nu_\beta a_n}\chi_2^+\right)\right]
 = \Et\left[\exp\left(\sum_{l=1}^\infty |\Ic_l|\log(\Psi_{l,0})\right)\right].
\end{flalign}
Furthermore, since $\tilde{\eta}_0\geq0$ we have that $\log(\Psi_{l,0}) \leq l\lambda\Eb[\eta_0]/a_n$.

Using a Taylor expansion we have that for any $x \in \Rb$ there exists $w\in[0,x]$ such that
\[e^x=1+x+\frac{e^w}{2}x^2.\]
It follows from this, and that $\log(\Psi_{l,0}) \leq l\lambda\Eb[\eta_0]/a_n$, that there exists a random variable $W$ which is bounded above by
\[ \frac{\lambda\Eb[\eta_0]}{a_n}\sum_{l=1}^\infty l|\Ic_l|\]
such that
\begin{flalign*}
  n\Et\left[\exp\left(\sum_{l=1}^\infty |\Ic_l|\log(\Psi_{l,0})\right)-1\right]
 \; =\; n\Et\left[\sum_{l=1}^\infty |\Ic_l|\log(\Psi_{l,0})+\frac{e^W}{2}\left(\sum_{l=1}^\infty |\Ic_l|\log(\Psi_{l,0})\right)^2\right].
\end{flalign*}
Using the previous upper bound on $\log(\Psi_{l,0})$ and Cauchy-Schwarz we then have that
\begin{flalign*}
n\Et\!\left[\frac{e^W}{2}\left(\sum_{l=1}^\infty |\Ic_l|\log(\Psi_{l,0})\right)^2\right] 
& \leq \frac{C_\lambda n}{a_n^2}\Et\!\left[\exp\left(\frac{\lambda\Eb[\eta_0]}{a_n}\sum_{l=1}^\infty l|\Ic_l|\right)\left(\sum_{l=1}^\infty l|\Ic_l|\right)^2\right] \\
 & \leq \frac{C_\lambda n}{a_n^2}\Et\!\left[\exp\left(\frac{2\lambda\Eb[\eta_0]}{a_n}\sum_{l=1}^\infty l|\Ic_l|\right)\right]^{1/2}\Et\!\left[\!\left(\sum_{l=1}^\infty l|\Ic_l|\right)^4\right]^{1/2}.
\end{flalign*}

The sum $\sum_{l=1}^\infty l|\Ic_l|$  is the time between regenerations of the embedded walk $Y$. By \cite[Lemma 5.1]{depeze96} this has exponential moments therefore we can choose $n$ sufficiently large such that both of the expectations in the previous equation are finite. Since $a_n=n^{1/\alpha}L(n)$ for $\alpha \in (1,2)$ we have that $n/a_n^2$ converges to $0$ as $\nin$ and therefore
\begin{flalign*}
  n\Et\left[\exp\left(\sum_{l=1}^\infty |\Ic_l|\log(\Psi_{l,0})\right)-1\right]= n\Et\left[\sum_{l=1}^\infty |\Ic_l|\log(\Psi_{l,0})\right]+o(1).
\end{flalign*}

The function $n\log(\Psi_{l,0})$ is deterministic and converges to the positive constant $f(l,\lambda)$. By Lemma \ref{l:UpLw} we have that for $n$ suitably large $cl\leq n\log(\Psi_{l,0})\leq Cl^2$ independently of $l$. It follows from (\ref{EI}) that 
\[\sum_{l=1}^\infty l^2\Et\left[|\Ic_l|\right]<\infty,\]
and therefore by dominated convergence we have that, as $\nin$, 
\begin{flalign*}
 n\Et\left[\sum_{l=1}^\infty |\Ic_l|\log(\Psi_{l,0})\right] \; = \; \sum_{l=1}^\infty \Et\left[|\Ic_l|\right] n\log(\Psi_{l,0}) \; \rightarrow \; \sum_{l=1}^\infty f(l,\lambda)\Et\left[|\Ic_l|\right]. 
\end{flalign*}
Combining with \eqref{e:eLogPsi} completes the proof.
\end{proof}
\end{prp}
  
By Lemma \ref{l:KillY} and Proposition \ref{p:DoA} we now have the following corollary.
\begin{cly}\label{c:regseq}   
Under the assumptions of Theorem \ref{t:fluc}, the sequence $\Sigma_{n_kt}/a_{n_k}$ converges in distribution on $D_{J_1}$ to a process $\tilde{\Vc}_t^\alpha$ with Laplace exponent $t\sum_{l=1}^\infty f(l,\lambda\nu_\beta)\Et[|\Ic_l|]$. 
\end{cly}

We now have the required framework to prove Theorem \ref{t:fluc} which is a convergence result for $(\Delta_{n_kt\nu_\beta}-n_kt)/a_{n_k}$ and $(X_{n_kt}-n_kt\nu_\beta)/a_{n_k}$. The approach we take is to compare $\Delta$ with $\Sigma$ and then use a result of \cite{wh02} that allows us to deduce convergence for the inverse of $\Delta$.

\begin{proof}[Proof of Theorem \ref{t:fluc}]
Write $\psi_t:=\sup\{j\geq 0:\zeta^X_j\leq t\}$ to be the number of regenerations by time $t$. It follows by the law of large numbers and \textit{continuity of the inverse at strictly increasing functions} \cite[Corollary 13.6.4]{wh02} that $\psi_{nt}n^{-1}$ converges $\Pb$-a.s.\ to $R_t:=(\Eb[\eta_0]\Eb[\zeta^Y_2-\zeta^Y_1])^{-1}t$. By Corollary \ref{c:regseq} and \textit{$J_1$-continuity of composition} \cite[Theorem 13.2.2]{wh02} we then have that $\Sigma_{\psi_{tn_k}}/a_{n_k}$ converges in distribution on $D_{J_1}([0,\infty),\Rb)$ to the process $\tilde{\Vc}_{R_t}$.

Recall that 
\[\Sigma_{m}=\left(\zeta^X_{m}\nu_\beta-X_{\zeta^X_{m}}\right) -\left(\zeta_1^X\nu_\beta-\varrho_1\right)\]
where $a_{n_k}^{-1}\left(\zeta_1^X\nu_\beta-\varrho_1\right)$ converges to $0$ in probability. It follows that
\begin{flalign*}
\Sigma_{\psi_{n_kt}}+ \left(\zeta_1^X\nu_\beta-\varrho_1\right) \; = \; \nu_\beta\zeta^X_{\psi_{n_kt}}-X_{\zeta^X_{\psi_{n_kt}}} \; = \; \nu_\beta\left(\Delta_{\varrho_{\psi_{n_kt}}}-\varrho_{\psi_{n_kt}}\nu_\beta^{-1}\right).
\end{flalign*}
Notice that $\varrho_{\psi_{n_kt}}n_k^{-1}$ is bounded above by $t$ and converges $\Pb$-a.s.\ to $t$ since the distance between regenerations has exponential moments by \cite[Lemma 5.1]{depeze96}. In particular, if 
\begin{flalign*}
u_k:=d_{M_1}\left(\left(\frac{\Delta_{\nu_\beta\varrho_{\psi_{n_kt}}}-\varrho_{\psi_{n_kt}}}{a_{n_k}}\right)_{t\in[0,T]},\left(\frac{\Delta_{n_kt\nu_\beta}-n_kt}{a_{n_k}}\right)_{t\in[0,T]}\right) 
\end{flalign*}
converges to $0$ in probability then $(\Delta_{n_kt\nu_\beta}-n_kt)/a_{n_k}$ converges in distribution on $D_{M_1}$ to the process $\Vc_t:=\nu_\beta^{-1}\tilde{\Vc}_{R_{\nu_\beta t}}$. 

Since $\Delta$ is increasing and the regeneration points are ordered we have that
\[\varrho_{\psi_{n_kt}}\leq n_kt \leq \varrho_{\psi_{n_kt}+1} \qquad \text{ and } \qquad \Delta_{\varrho_{\psi_{n_kt}}}\leq \Delta_{n_kt} \leq \Delta_{\varrho_{\psi_{n_kt}+1}}. \]
In particular, $u_k\leq (1+\nu_\beta)(\varrho_{\psi_{n_kt}+1}-\varrho_{\psi_{n_kt}})/a_{n_k}$. Using that there are at most $n_kT$ regenerations by level $n_kT$ we then have that $u_k$ is bounded above by $C\sup_{j\leq n_kT}(\varrho_{j+1}-\varrho_j)/a_{n_k}$. Let $\varepsilon>0$, using a union bound and that the distance between regenerations have exponential moments we have that
\begin{flalign*}
 \Pb\left(u_k >\varepsilon\right) \;\leq \; n_kT\Pt(\varrho_2-\varrho_1>C\varepsilon a_{n_k}) +\Pt(\varrho_1>C\varepsilon a_{n_k}) \; \leq \; C_Tn_ke^{-c_\varepsilon a_{n_k}}
\end{flalign*}
which converges to $0$ as $k\rightarrow \infty$.

The inverse $\Delta_{t}^{-1}=\sup\{X_s: \; s\leq t\}=:\overline{X}_t$ is the furthest vertex reached by time $t$. By \textit{inverse with linear centring} \cite[Theorem 13.7.1]{wh02} we have that
\[\left(\frac{\Delta_{n_kt\nu_\beta}-n_kt}{a_{n_k}}, \frac{\overline{X}_{n_kt}-n_kt\nu_\beta}{a_{n_k}}\right)_{t\geq 0}\]
 converge jointly in distribution on $D_{M_1}\times D_{M_1}$ to the pair $(\Vc_t,-\nu_{\beta}\Vc_{t})$. Moreover, letting $\tau_0$ denote the first hitting time of $0$, using a union bound we have
\begin{flalign*}
 \Pb\left(\sup_{t\leq nT}\overline{X}_t-X_t>C\log(n)\right) \;\leq\; nT\Pt_{\lfloor C\log(n)\rfloor}(\tau_0<\infty) \;=\; nT\beta^{-\lfloor C\log(n)\rfloor}
\end{flalign*}
which converges to $0$ for $C$ large therefore we have the desired convergence replacing $\overline{X}$ with $X$.

Finally, as in the proof of Theorem \ref{t:Sstab}, independent and stationary increments of $\Vc_t$ follow from independent and stationary increments of the discrete process. Combined with the form of the Laplace transform we have that $\Vc_t$ is a L\'{e}vy process. Similarly, when the convergence holds along all subsequences $n_k$, the form of the Laplace transform yields self-similarity and therefore the process is $\alpha$-stable.
\end{proof}

Noting that $\hat{\Vc}_t:=-\tilde{\Vc}_{R_{\nu_{\beta}t}}$ is the limiting process for $(X_{n_kt}-n_kt\nu_\beta)/a_{n_k}$ where $R_t=(\Eb[\eta_0]\Eb[\zeta^Y_2-\zeta^Y_1])^{-1}t$ and $\tilde{\Vc}_{t}^\alpha$ has Laplace exponent $t\sum_{l=1}^\infty f(l,\lambda \nu_\beta)\Et[|\Ic_l|]$, we can express the moment generating function of $\hat{\Vc}^\alpha_t$ as
\begin{flalign*}
\Eb[e^{\lambda\hat{\Vc}_t}]\;=\;\exp\left(\frac{t\nu_\beta}{\Eb[\eta_0]\Eb[\zeta^Y_2-\zeta^Y_1]}\sum_{l=1}^\infty f(l,\lambda\nu_\beta)\Et[|\Ic_l|]\right).
\end{flalign*}
Notice that the random variables $\eta_m-\Eb[\eta_0]$ are bounded below and therefore have a one-sided heavy tail. It follows that $\Delta_{nt}-nt\nu_{\beta}$ is a totally asymmetric jump process and therefore the process $\Vc_t$ can be decomposed into a positive jump process with negative drift. 

We now prove a technical result that will make the condition \eqref{LogAss} more straightforward to apply. We show that we can decompose the excursion times and remove parts of the excursion that do not contribute to the fluctuations. Let $\tilde{X}$ be a randomly trapped random walk with some trap measure $\tilde{\pi}$ and embedded walk $Y$ then denote by $\tilde{\eta}$ the sequence of holding times.
\begin{cly}\label{c:ign}
Let $\beta>1$, $\alpha\in(1,2)$ and $f:\Nb\times\Rb_+\rightarrow \Rb$ be such that $f(1,\lambda)\asymp\lambda^\alpha$. Suppose that the exists a coupling of $X$, $\tilde{X}$ such that, for some $\epsilon>0$, $\Eb[|\eta_0-\tilde{\eta}_0|^{\alpha+\epsilon}]<\infty$. If there exists $n_k\nearrow \infty$, $a_n=n^{1/\alpha}L(n)$ for some $L$ slowly varying such that for any $l \in\Nb$ and $\lambda>0$ we have
  \begin{flalign*}
  n_k\log\left(\Er\left[\Et^\omega\left[\exp\left(-\frac{\lambda}{a_{n_k}}(\tilde{\eta}_0-\Eb[\tilde{\eta}_0])\right)\right]^l\right]\right)\sim f(l)\lambda^\alpha
  \end{flalign*}
as $k\rightarrow \infty$ then, as $k\rightarrow \infty$, 
\[\left(\frac{\Delta_{n_kt\nu_\beta}-n_kt}{a_{n_k}}, \frac{X_{n_kt}-n_kt\nu_\beta}{a_{n_k}}\right)_{t\geq 0}\]
 converges in $\Pb$-distribution on $D_{M_1}\times D_{M_1}$ to $(\Vc_t, -\nu_{\beta}\Vc_t)_{t\geq 0}$ where $\Vc_t$ is a L\'{e}vy process.
In particular, if \eqref{LogAss} holds for all subsequences $n_k$ then $f(l,\lambda)=g(l)\lambda^\alpha$ for some $g\not\equiv0$ and $\Vc_t$ is an $\alpha$-stable process. 
\end{cly}
\begin{proof}
Since, in each model, the traps are i.i.d.\ under $\Pr$ and each holding time in a fixed trap is independent, by translation invariance there exists a version $\tilde{X}$ coupled to $X$ such that $\Eb[(\eta_k-\tilde{\eta}_k)^{\alpha+\epsilon}]<\infty$ for all $k \in \Nb$. Write $\tilde{\Delta}_m:=\inf\{t>0:\tilde{X}_t=m\}$ to be the first hitting time of level $m$ by the walk $\tilde{X}$ and $\tilde{\nu}_\beta$ its speed. We want to show that
\begin{equation}\label{e:Diff}
\sup_{m\leq n}\left|\frac{\tilde{\Delta}_m-\Delta_m-m(\tilde{\nu}_\beta^{-1}-\nu_\beta^{-1})}{a_n}\right|
\end{equation}
converges to $0$ in probability as $\nin$.

First suppose that $\tilde{\eta}_k\geq \eta_k$ $\Pb$-a.s.. In particular, we assume that $\Eb[\tilde{\eta}_0-\eta_0]>0$ since, otherwise, the result follows trivially. Define $\hat{X}$ to be the randomly trapped random walk with embedded walk $Y$ and holding times $\tilde{\eta}_k-\eta_k$. This walk has first hitting times $\hat{\Delta}_m=\tilde{\Delta}_m-\Delta_m$ and, since $\Eb[\hat{\eta}_0^{\alpha+\epsilon}]<\infty$, we have that the walk is ballistic and has speed $\hat{\nu}_\beta=\upsilon_\infty/\Eb[\tilde{\eta}_0-\eta_0]$ (see \eqref{e:spd}) which gives us that $\hat{\nu}_\beta^{-1}=\tilde{\nu}_\beta^{-1}-\nu_\beta^{-1}$. In particular \eqref{e:Diff} is equal to 
\begin{equation*}
\sup_{m\leq n}\left|\frac{\hat{\Delta}_m-m\hat{\nu}_\beta^{-1}}{a_n}\right|.
\end{equation*}
It then follows straightforwardly from the proof of Theorem \ref{t:fluc} that this converges to $0$ in probability as $\nin$. 

Note that by a symmetric argument we have that if $\tilde{\eta}_k\leq \eta_k$ $\Pb$-a.s.\ then \eqref{e:Diff} converges to $0$ in probability as $\nin$. For general $\tilde{X}$ we can define a new randomly trapped random walk with holding times $\tilde{\eta}_k\land \eta_k$. Using the triangle inequality we then have that \eqref{e:Diff} converges to $0$ in probability as $\nin$. The result then follows by Theorem \ref{t:fluc}.
\end{proof}

\section{Counterexample for convergence in the Skorohod topology}\label{s:non}
Recall that in Theorem \ref{t:fluc} we consider the $M_1$ topology as opposed to the stronger $J_1$ topology considered for $\Sigma$ in Corollary \ref{c:regseq}. Given that our approach was to compare $\Delta$ with $\Sigma$, one may question whether we should be able to extend the statements of Theorem \ref{t:fluc} to the $J_1$ topology. This is not possible in the general setting because, between regenerations, the walk experiences several large holding times at the same significant vertex. This means that $\Delta$ fluctuates between regeneration points whereas $\Sigma$ groups the fluctuations into a single jump. 
   
We now use the example of the random walk in random scenery to show that, in general, the convergence in Theorem \ref{t:fluc} cannot be extended to $J_1$ convergence. Let $(\kappa_x)_{x\in \Zb}$ be an i.i.d.\ sequence of $(0,\infty)$-valued random variables under $\Pr$. We then define the holding times by $\Pt^\omega(\eta_{x,j}=\kappa_x)=1$ for all $j\geq 1$ and $x \in \Zb$. Let $X$ denote the randomly trapped random walk, $S$ its clock process and $\Delta$ the first hitting times. 

Proposition \ref{p:RWRS} shows that the assumptions of Theorems \ref{t:Sstab} and \ref{t:fluc} hold under a simple stable law condition for the environment. 
\begin{prp}\label{p:RWRS}
Let $\alpha >0$, $\beta>1$ and suppose that $\Pr(\kappa_0\geq t) \sim t^{-\alpha}\overline{L}(t)$ as $t\rightarrow \infty$ for a function $\overline{L}$ which varies slowly at $\infty$. Then, there exists $a_n=n^{1/\alpha}L(n)$ for a slowly varying function $L$ such that
\begin{enumerate}
 \item If $\alpha \in (0,1)$ then as $n\rightarrow \infty$, $(S_{nt}/a_{n},X_{a_{n}t}/n)_{t\geq 0}$ converges in $\Pb$-distribution on $D_{M_1}\times D_{U}$ to $(\Sc_t,\upsilon_\beta\Sc^{-1}_t)_{t\geq 0}$ where $\Sc_t$ is an $\alpha$-stable subordinator. 
 \item If $\alpha \in (1,2)$ write $\nu_\beta=\upsilon_\beta/\Er[\kappa_0]$ then, as $n\rightarrow \infty$, 
\[\left(\frac{\Delta_{nt\nu_\beta}-nt}{a_{n}}, \frac{X_{nt}-nt\nu_\beta}{a_{n}}\right)_{t\geq 0}\]
 converges in $\Pb$-distribution on $D_{M_1}\times D_{M_1}$ to $(\Vc_t, -\nu_{\beta}\Vc_t)_{t\geq 0}$ where $\Vc_t$ is an $\alpha$-stable process. 
\end{enumerate}
\begin{proof}
The holding time $\eta_0$ is almost surely fixed and equal to $\kappa_0$ under $\Pt^\omega$; we therefore have that $\Et^\omega[g(\eta_0)]=g(\kappa_0)$ for any function $g$. 

For $\alpha \in (0,1)$, by assumption, $\kappa_0$ belongs to the domain of attraction of stable law of index $\alpha$ with respect to $\Pr$ hence there exists $a_n=n^{1/\alpha}L(n)$ such that $n\Pr(\kappa_0>ta_n)\sim Ct^{-\alpha}$. Since $\kappa_0$ is almost surely non-negative, it follows from a Tauberian theorem (e.g.\ \cite[Theorem XIII.5.1]{fe71}) that
\begin{flalign*}
 -n\log\left(\Er\left[\Et^\omega\left[\exp\left(-\frac{\lambda}{a_n}\eta_0\right)\right]^l\right]\right) \;=\; -n\log\left(\Er\left[\exp\left(-\frac{l\lambda}{a_n}\kappa_0\right)\right]\right) \; \sim \; C(l\lambda)^\alpha.
\end{flalign*}
The first result then follows from Theorem \ref{t:Sstab}.

Similarly, for $\alpha \in (1,2)$, by assumption, $\kappa_0$ belongs to the domain of attraction of stable law of index $\alpha$ with respect to $\Pr$ hence there exists $a_n=n^{1/\alpha}L(n)$ such that $n\Pr(\kappa_0-\Er[\kappa_0]>ta_n)\sim Ct^{-\alpha}$ and therefore, since $\kappa_0-\Er[\kappa_0]$ is centred and bounded below,
\begin{flalign*}
n\log\left(\!\Er\left[\Et^\omega\!\left[\exp\left(-\frac{\lambda}{a_n}(\eta_0-\Eb[\eta_0])\right)\right]^l\right]\right) = n\log\left(\!\Er\left[\exp\left(-\frac{l\lambda}{a_n}(\kappa_0-\Er[\kappa_0])\right)\right]\right) 
\end{flalign*}
converges to $C(l\lambda)^\alpha$ for some constant $C$. The result then follows from Theorem \ref{t:fluc}.
\end{proof}
\end{prp}

We now give an argument which shows that, in general, the $M_1$ convergence in Theorem \ref{t:fluc} cannot be extended to $J_1$ convergence. For $f \in D([0,\infty),\Rb)$ and $T,h>0$ let
\[\tilde{\omega}(f,T,h):=\inf_{(I_k)}\max_k\sup_{r,s \in I_k}|f(s)-f(r)|\]
denote the modulus of continuity where the infimum is over partitions of $[0,T)$ into $I_k=[v_k,v_{k+1})$ such that $v_{k+1}-v_k>h$ for all $k\geq 1$. By \cite[Theorem 16.10]{ka02}, if $Z,(Z^n)_{n=1}^\infty$ are random elements of $D([0,\infty),\Rb)$ then $Z^n$ converges in distribution to $Z$ in $D_{J_1}$ if and only if $Z^n$ converges to $Z$ in the sense of finite dimensional distributions and, for any $T>0$,
\[\lim_{h\rightarrow 0}\limsup_{\nin}\Eb[\tilde{\omega}(Z^n,T,h)\land 1]=0.\]

Let us assume that $\beta>1$ and $\Pr(\kappa_0\geq t)  = t^{-\alpha}\land 1$ for $\alpha \in (1,2)$. Write 
\[Z^n(t):=\frac{X_t-t\nu_\beta}{a_n} \quad \text{ then } \quad |Z^n(s)-Z^n(r)|=\left|\frac{X_s-X_r-\nu_\beta(s-r)}{a_n}\right|.\]

Since $\alpha>1$ we have that the walk is ballistic therefore 
\[\limn\Pb\left(\Delta_{nT\nu_\beta/2}\leq nT\right)=1.\]
That is, with high probability the walk visits at least $nT\nu_\beta/2$ vertices by time $nT$. 

Let $\lambda>0$ then the probability that at least one of these first $nT\nu_\beta/2$ vertices has a holding time of at least $\lambda_n:=\lambda a_n=\lambda n^{1/\alpha}$ and no more than $h_n:=hn/3$ is at least
\begin{flalign*}
 \Pb\left(\bigcup_{x=1}^{nT\nu_\beta/2}\!\!\!\{\kappa_x\in [\lambda_n,h_n)\}\!\right)  = 1-\left(1-(\lambda_n)^{-\alpha}+(h_n)^{-\alpha}\right)^{nT\nu_\beta/2}  \rightarrow 1-e^{-\frac{\lambda^{-\alpha}T\nu_\beta}{2}}>0. 
\end{flalign*}

Let $\hat{x}:=\inf\{x\geq 0:\kappa_x \in [\lambda_n,h_n)\}$. Notice that the probability that the embedded walk moves from a vertex $y$ to $y-1$ and then back to $y$ starting from the first hitting time of $y$ is 
\[\Pb(Y_{\Delta_y^Y+1}=y-1, \; Y_{\Delta_y^Y+2}=y)=\frac{\beta}{(\beta+1)^2}>0\]
independently of the environment. In particular, this holds for $y=\hat{x}$ and, writing 
\begin{flalign*}
\Ac_n  & :=\{\Delta_{nT\nu_\beta/2}\leq nT\}\;\cap\;\{\hat{x}\leq nT\nu_\beta/2\} \;\cap\;\{Y_{\Delta_{\hat{x}}^Y+1}=\hat{x}-1\}\;\cap\;\{Y_{\Delta_{\hat{x}}^Y+2}=\hat{x}\},
\end{flalign*}
we have that for any fixed $h>0$ 
\[\limsup_{\nin}\Pb(\Ac_n)\;\geq\; \frac{\beta}{(\beta+1)^2}\left(1-e^{-\frac{\lambda^{-\alpha}T\nu_\beta}{2}}\right)\;>\;0.\] 

Let $\vartheta_n:= \Delta_{\hat{x}}$ be the first hitting time of $\hat{x}$ then $\vartheta_n^+:= \vartheta_n+\kappa_{\hat{x}}$ is the time that the walk first moves from $\hat{x}$ to one of its neighbours. Similarly, write $\tilde{\vartheta}_n:= \inf\{t>\vartheta_n^+:X_t=\hat{x}\}$ to be the first return time to $\hat{x}$ after the first visit and let $\tilde{\vartheta}_n^+:= \tilde{\vartheta}_n+\kappa_{\hat{x}}$ be the time at which the walk moves away from $\hat{x}$ for the second time. 

Let $(I_k)$ be a partition of $[0,T)$ into intervals $[v_k,v_{k+1})$ satisfying $v_{k+1}-v_k>h$ for all $k$. On the event $\Ac_n$ we have that for any such partition there exists some $k$ such that either $\vartheta_n,\vartheta_n^+ \in I_k$ or $\tilde{\vartheta}_n,\tilde{\vartheta}_n^+ \in I_k$. We then have that 
\[\limsup_{\nin}|Z^n(\vartheta_n)-Z^n(\vartheta_n^+)|\;=\;\limsup_{\nin}\frac{\nu_\beta\kappa_{\hat{x}}+1}{a_n}\geq \lambda \nu_\beta.\]
The same holds for $\tilde{\vartheta}_n,\tilde{\vartheta}_n^+$. In particular, choosing $\lambda \leq \nu_\beta^{-1}$,
\[\limsup_{\nin}\Eb[\tilde{\omega}(Z^n,T,h)\land 1]\;\geq\; \lambda \nu_\beta\limsup_{\nin}\Pb(\Ac_n)\; \geq\; \frac{\lambda \nu_\beta\beta\left(1-e^{-\frac{\lambda^{-\alpha}T\nu_\beta}{2}}\right)}{(\beta+1)^2}\;>\;0.  \]
This proves that 
\[\lim_{h\rightarrow 0}\limsup_{\nin}\Eb[\tilde{\omega}(Z^n,T,h)\land 1]\neq0.\]
for the random walk in random scenery and, therefore, Theorem \ref{t:fluc} cannot be extended to the $J_1$ topology. 

\section{Random walks on Galton-Watson trees}\label{s:GWTree}
For a fixed tree $\Ts$ rooted at $\rho$, we write $\overleftarrow{x}$ to denote the parent of $x \in \Ts$ and $c(x)$ the set of children of $x$. For $\beta\geq 1$, we then define the $\beta$ biased random walk $X$ as the Markov chain started from a fixed vertex $z$ with transition probabilities 
\begin{flalign*}
\Pt^{\Ts}_z(X_{n+1}=y|X_n=x)=\begin{cases} \frac{1}{1+\beta |c(x)|}, & \text{if } y=\overleftarrow{x}, \\  \frac{\beta}{1+\beta |c(x)|}, & \text{if } y \in c(x), \; x \neq \rho, \\ \frac{1}{
|c(\rho)|}, & \text{if } y \in c(x), \; x =\rho, \\ 0, & \text{otherwise.} \\ \end{cases} 
\end{flalign*}
That is, $X$ is the random walk which is $\beta$ times more likely to move along a given edge away from the root than it is to move along the unique edge directed towards the root.

Let $\xi$ be a random variable taking values in the non-negative integers with masses $p_k:=\Pr(\xi=k)$, mean $\mu\in(0,1)$ and variance $\sigma^2 <\infty$. Denote by $Z_n$ the $n\th$ generation size of a GW-process with this offspring law (e.g.\ \cite{atne04}). Such a process gives rise to a random rooted tree $\Tc^\xi$ where individuals in the process are represented by vertices (with the unique progenitor as the root $\rho$) and undirected edges connect individuals with their offspring. To avoid the trivial case in which no traps form we also assume that $p_0+p_1<1$. 

It has been shown in \cite{ke86} that there is a well defined probability measure $\Pr$ over GW-trees conditioned to survive $\Tc$ which arises as the limit as $\nin$ of probability measures over GW-trees conditioned to survive up to generation $n$. It can be seen (e.g.\ \cite{ja12}) that the tree can be constructed by attaching i.i.d.\ finite trees to a single semi-infinite path. More specifically, we can construct a random tree with the desired law using two types of vertices which we refer to as special and normal. Define $\xi^*$ to be a random variable with the size-biased law given by the probabilities $\Pr(\xi^*=k)=kp_k\mu^{-1}$. The construction is then as follows:
\begin{enumerate}[1)]
\item
Start with a single special vertex;
\item 
At each generation, let every normal vertex give birth to vertices according to independent copies of the original offspring distribution (all of which are labelled as normal);
 \item 
At each generation, let every special vertex give birth to vertices according to independent copies of the size-biased distribution (one of which is chosen uniformly at random to be special and the remaining vertices are labelled as normal).
\end{enumerate} 
We will denote by $\rho_k$ the special vertex in the $k\th$ generation and $\Yc:=(\rho_0=\rho,\rho_1,...)$ the semi-infinite path formed by the special vertices (which we call the backbone). We then write $\Tc^{*-}_x$ for the branch rooted at $x \in \Yc$; that is, the tree formed by the descendants of $x$ which are not in the descendant tree of the unique child of $x$ on $\Yc$ (see Figure \ref{f:treediag}). 

\begin{figure}[t]
\centering
 \includegraphics[scale=0.8]{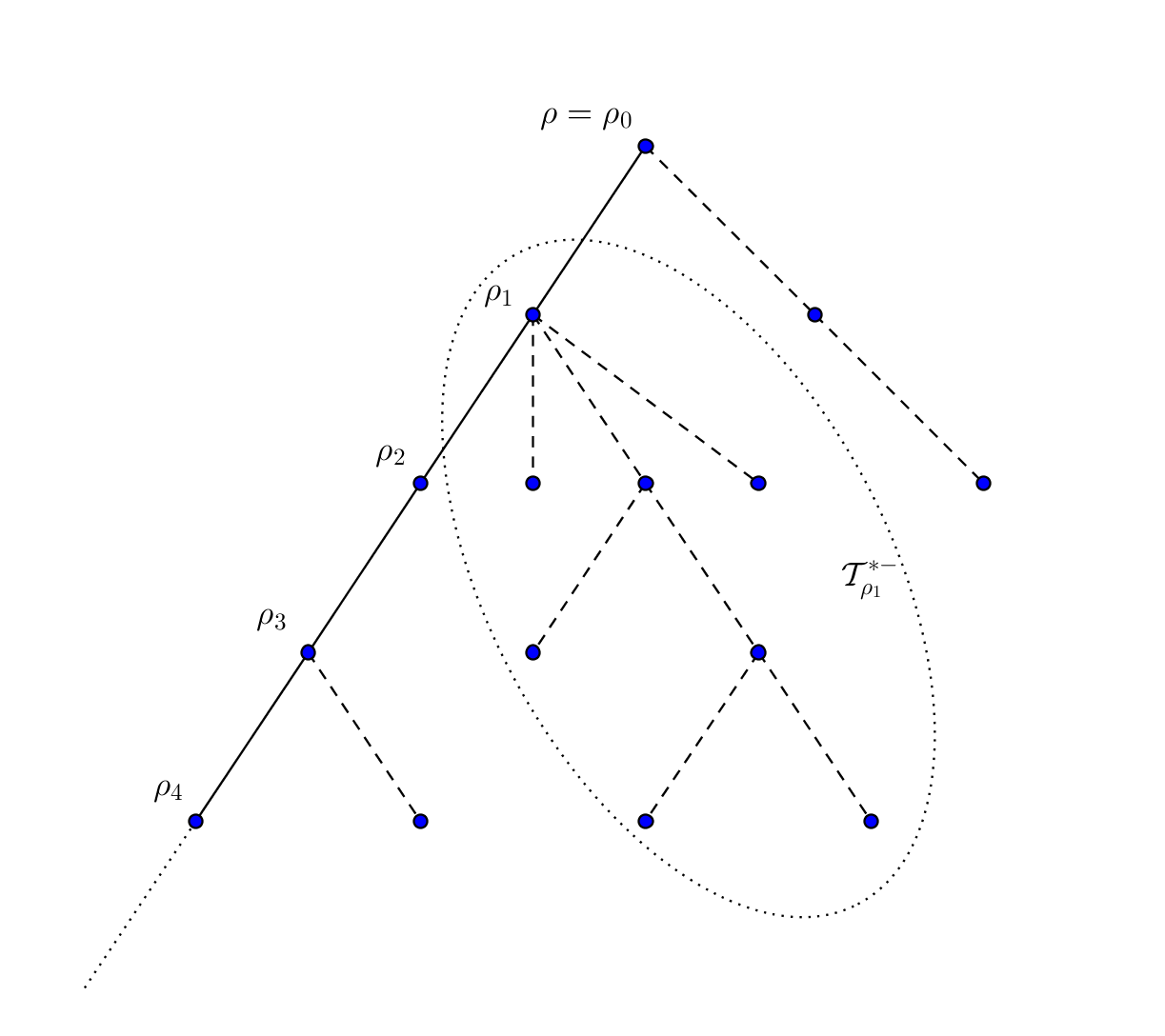} 
\caption{A sample subcritical GW-tree conditioned to survive $\Tc$ with the backbone $\Yc$ represented by solid lines and the buds and traps connected by dashed lines.}\label{f:treediag}
\end{figure}
In this section we consider a biased random walk $X$ on $\Tc$ and we denote by $|x|:=d(x,\rho)$ the graph distance between the $x\in\Tc$ and the root of the tree. As for the randomly trapped random walk, we use $\Pb(\cdot)=\int \Pt^{\Tc}_\rho(\cdot)\Pr(\text{d}\Tc)$ for the annealed law.

%For a fixed $\rho$-rooted tree $\Ts$ associated to $Z_n$ with $Z_1>0$ it follows from hitting time identities in electrical network theory (e.g.  that 
% \begin{flalign}\label{l:return}
% \Et^{\Ts}_\rho[\tau^+_\rho]=2\sum_{n\geq 1}\frac{Z_n\beta^{n-1}}{Z_1}
% \end{flalign}
%where $\tau^+_x:=\inf\{k>0:X_k=\rho\}$ is the first return time to a vertex $x$ by a $\beta$ biased random walk $X$. 

It has been shown in \cite[Theorem 4]{bo18_1} that if $1<\beta<\mu^{-1}$ then $|X_n|/n$ converges $\Pb$-a.s.\ to a fixed (known) speed $\nu_\beta$. Moreover, this has been extended to a functional central limit theorem in \cite[Theorem 4]{bo18_1} when $1<\beta<\mu^{-1/2}$ and $\Er[\xi^3]<\infty$. That is, under these conditions, $(|X_{nt}|-nt\nu_\beta)/\sqrt{n}$ converges in $\Pb$-distribution to a scaled Brownian motion. In this section we consider the case $\mu^{-1/2}<\beta<\mu^{-1}$ in which the walk is ballistic and does not satisfy a central limit theorem. We show in Theorem \ref{t:GWT} that, along given subsequential limits, the suitably centred and scaled process converges to a L\'{e}vy process. In particular, the correct scaling of the centred walk up to time $n$ is $n^{1/\gamma}$ where
\begin{flalign*}
\gamma:=-\frac{\log(\mu)}{\log(\beta)}\in(1,2).
\end{flalign*} 
\begin{thm}\label{t:GWT}
Suppose $\Er[\xi^{2+\gamma}]<\infty$ and $\mu^{-1/2}<\beta<\mu^{-1}$ then, for $\varsigma>0$ write $n_k(\varsigma)=\lfloor \varsigma\mu^{-k} \rfloor$. As $k\rightarrow \infty$, 
\[\left(\frac{\Delta_{n_kt\nu_\beta}-n_kt}{n_k^{1/\gamma}}, \frac{X_{n_kt}-n_kt\nu_\beta}{n_k^{1/\gamma}}\right)_{t\geq 0}\]
 converges in $\Pb$-distribution on $D_{M_1}\times D_{M_1}$ to $(\Vc_t, -\nu_{\beta}\Vc_t)_{t\geq 0}$ where $\Vc_t$ is a L\'{e}vy process.
\end{thm}

It is natural to consider whether the convergence can be proved more generally (i.e.\ along all subsequences) however, this is not the case.
To see this we first note that the height $\Hc:=\sup\{n\geq 0:Z_n>0\}$ of the subcritical tree is approximately geometrically distributed with parameter $\mu$; specifically, by \cite[Theorem B]{lypepe95} the sequence $\Pr(\Hc\geq n)\mu^{-n}$ is decreasing and converges to a positive constant:
\begin{flalign}\label{l:cmu}
c_\mu:=\limn \frac{\Pr(\Hc\geq n)}{\mu^{n}}.
\end{flalign}
The time spent in a branch largely consists of a geometric number of excursions from the deepest point to itself. Each of these excursions is short but, by the Gambler's ruin, the probability that the walk escapes the branch on a given excursion is of the order $\beta^{-\Hc}$. It follows that the trapping times behave approximately like $\beta^{\Hc}$. These times do not belong to the domain of attraction of any stable law and therefore we do not observe full convergence. We direct the reader to \cite{arfrgaha12} for a more detailed illustration of this lattice effect. A related model has been considered in \cite{arha12} and \cite{ha13} in which the bias is chosen randomly for each edge according to a non-lattice condition. This modification removes the lattice effect in the sub-ballistic regime and should also be applicable here.

\begin{figure}[b!]
\centering
 \includegraphics[scale=0.4]{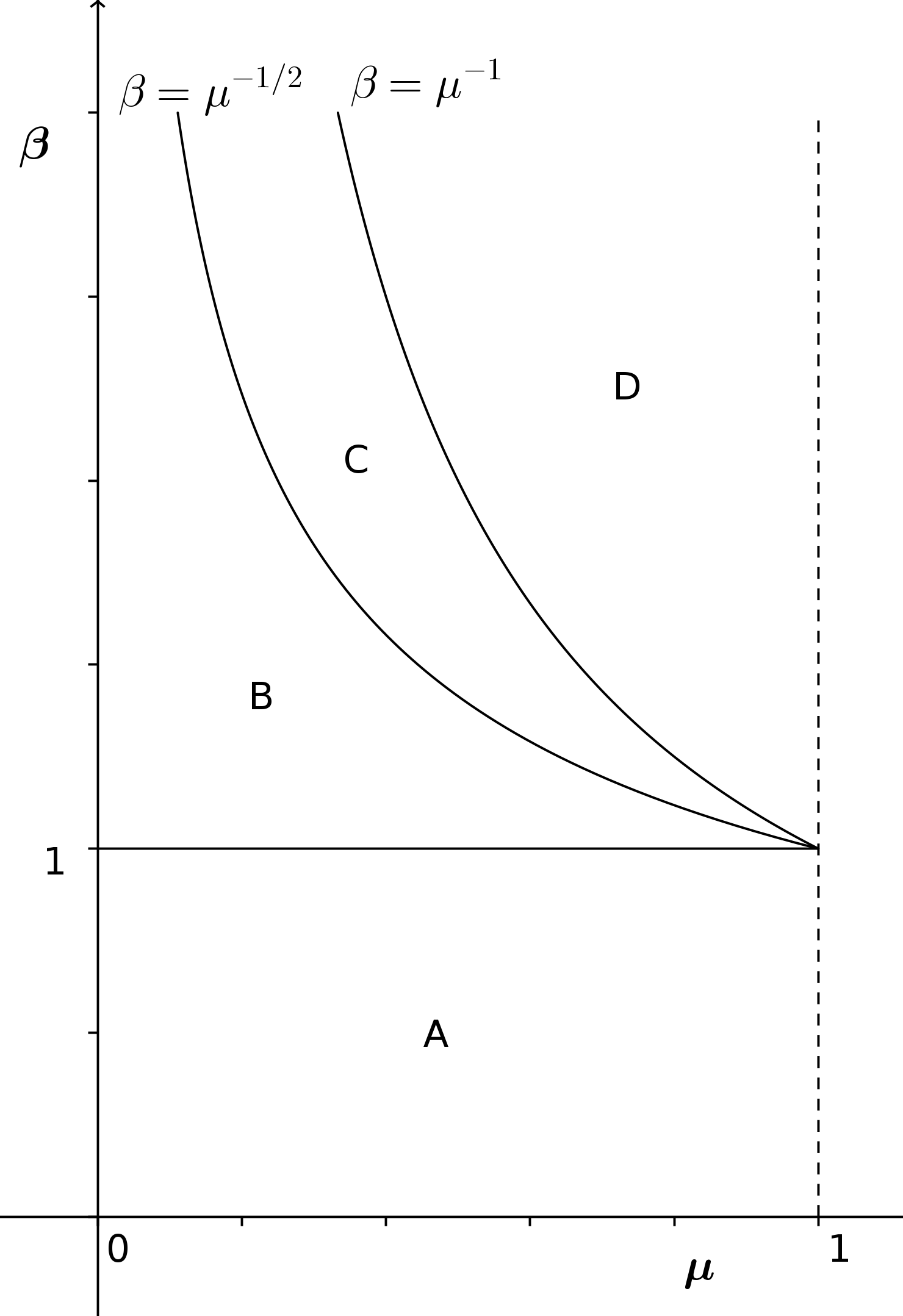} 
\caption{Phase diagram depicting the escape regimes of the biased random walk on the subcritical Galton-Watson tree for different combinations of mean and bias. In region A the walk is recurrent. In region B the walk is ballistic and satisfies a functional central limit theorem. In region C the walk is ballistic and has fluctuations of order $n^{1/\gamma}$. In region D the walk is subballistic and the distance reached by the walker by time $n$ is of order $n^\gamma$ (meaning level $n$ is first reached around time $n^{1/\gamma}$).}\label{f:phase}
\end{figure}

The case in which $\beta>\mu^{-1}$ has been studied further in \cite{bo18} where it is shown that $|X_n|$ scales with $n^\gamma$ but also experiences the lattice effect. This corresponds with the regime covered by Theorem \ref{t:Sstab} which we do not study further here. Furthermore, similar results to those mentioned above have been shown for the random walk on critical and supercritical GW-trees (cf.\ \cite{ai14, arfrgaha12, bo18_2, crfrku13, lypepe96, peze08}). We also refer the reader to  \cite{lype16}) for further background on random walks on GW-trees and their connection to electrical network theory. 

%The main aim of this section is to prove the following result which determines the correct scaling of the fluctuations for a random walk on a subcritical GW-tree.  
%\begin{thm}\label{t:GWT}
%Suppose $\Er[\xi^{2+\gamma}]<\infty$ and $\mu^{-1/2}<\beta<\mu^{-1}$ then, for $\varsigma>0$ write $n_k(\varsigma)=\lfloor \varsigma\mu^{-k} \rfloor$. As $k\rightarrow \infty$, 
%\[\left(\frac{\Delta_{n_kt\nu_\beta}-n_kt}{n_k^{1/\gamma}}, \frac{X_{n_kt}-n_kt\nu_\beta}{n_k^{1/\gamma}}\right)_{t\geq 0}\]
% converges in $\Pb$-distribution on $D_{M_1}\times D_{M_1}$ to $(\Vc_t, -\nu_{\beta}\Vc_t)_{t\geq 0}$ where $\Vc_t$ is a L\'{e}vy process.
%\end{thm}

Although Theorem \ref{t:GWT} does not extend to random walks on supercritical GW-trees, many of the techniques used in the proof of Theorems \ref{t:fluc} and \ref{t:GWT} can be adapted to the supercritical setting. In particular, a corresponding result should hold with $\beta_c^{1/2}<\beta<\beta_c$ and $\gamma=\log(\beta_c)/\log(\beta)$ where $\beta_c$ is the critical value of the bias determined in \cite{lypepe96}. Proving this would confirm \cite[Conjecture 3.3]{arfr16}.

The remainder of this section will proceed as follows. In Subsection \ref{s:TtR} we show that we can consider the walk on the tree in the RTRW framework. This is done by extending the GW-tree conditioned to survive into a two-sided tree, constructing holding times for a RTRW using a sequence of GW-trees and showing that the finite trees are sufficiently short so that the walk never deviates too far from the backbone. In Subsection \ref{s:DTT} we decompose the trapping times and use Corollary \ref{c:ign} to remove parts of the excursions which do not contribute to the limiting process. Finally, in Subsection \ref{s:ARTRW}, we show that the asymptotic \eqref{LogAss} holds with these reduced excursion times. This allows us to deduce the result from Theorem \ref{t:fluc}.

\subsection{The walk on the tree as a randomly trapped random walk}\label{s:TtR}
In this section we show that it suffices to consider a certain RTRW instead of the walk on a GW-tree. We begin by extending the GW-tree to a two sided version. A two-sided tree $\Tc^*$ can be constructed as the extension of a subcritical GW-tree conditioned to survive by using the infinite backbone $\Yc^*=(...,\rho_{-1},\rho_0,\rho_1,...)$ and i.i.d.\ branches. We then write $\Tc$ for the maximal subtree of $\Tc^*$ rooted at $\rho_0$. Let $X^*$ be a random walk on $\Tc^*$ and $X$ the walk on $\Tc$ coupled to $X^*$ so that 
\[X_n=X^*_{\Ls(n)} \quad \text{where} \quad \Ls(n)=\min\left\{m\geq 0:\sum_{k=1}^m\ind_{\{X^*_{k-1},X^*_k\in\Tc\}}\geq n\right\}.\]
The process $X$ is then a biased random walk on $\Tc$ whose transitions coincide with the transitions of $X^*$ restricted to $\Tc$. The two walks $X$ and $X^*$ deviate by at most the time spent by $X^*$ in $\Tc^*\setminus\Tc$. By transience of the embedded walk on $\Yc^*$, the walk $X^*$ almost surely spends a finite amount of time outside the infinite sub-tree rooted at $\rho_0$ (cf.\ \cite{bo17}). In particular, $d(X_n,X^*_n)<\log(n)$ with high probability therefore, since we consider polynomial scaling, we can consider the walk $X^*$ instead of $X$. 

We now construct the holding times of the randomly trapped random walk via a sequence of i.i.d.\ trees. For $x\in\Yc^*$, let $\Tc^{*-}_x$ be the branch rooted at $x$ in $\Tc^*$. Denote by $\oT_x$ the finite tree $\Tc^{*-}_x$ with an additional vertex $\oR_x$ (germ) appended as a parent of the root $x$. Transitions to this additional vertex will correspond to transitions along the backbone.

% Start with an initial vertex $\rho$ (root) and a unique ancestor $\oR$ (germ). Attach $\xi^*-1$ offspring to $\rho$ where $\xi^*$ is size-biased then attach independent GW-trees to the offspring of $\rho$. This creates a tree $\oT$ which has the distribution of a branch with an additional vertex connected to the root. Transitions to this additional vertex will correspond to transitions along the backbone.
%
%\begin{figure}[t]
%\centering
% \includegraphics[scale=0.8]{Figures/oLeftBch.png} 
%\caption{A tree $\oT$ with fixed vertices $\rho,\oR$ and $\xi^*-1$ independent GW-trees attached to $\rho$.}
%\end{figure}

Consider a walk $W$ on $\oT_x$ with transition probabilities 
\begin{flalign}\label{e:W}
\Pt^{\oT_x}(W_{n+1}=y|W_n=w)=
\begin{cases} 
1, & \text{if } x=y=\oR, \\
\frac{\beta+1}{1+\beta (|c(w)|+1)}, & \text{if } w=x, y=\oR_x, \\
\frac{\beta}{1+\beta (|c(w)|+1)}, & \text{if } w=x, y \in c(w), \\
\frac{\beta}{1+\beta|c(w)|}, & \text{if } w \notin\{x,\oR_x\}, y \in c(w),  \\
\frac{1}{1+\beta|c(w)|}, & \text{if } w \notin\{x,\oR_x\}, y =\overleftarrow{w},  \\
0, & \text{otherwise.} \\
\end{cases} 
\end{flalign} 
An excursion by $W$ in $\oT_x$ started from $x$ until absorption in $\oR_x$ has the same distribution as the time taken to move between backbone vertices of $\Tc^*$ started from $x$. In particular, we can construct a randomly trapped random walk on $\Zb$  with embedded walk given by the projection of $Y^*$ (the embedded walk of $X^*$) onto $\Zb$, environment determined by $\omega=(\oT_{\rho_k})_{k \in \Zb}$ and excursion times given by the excursions of $X^*$. In particular, we have that the RTRW is at $k \in \Zb$ precisely when $X^*$ is in $\Tc^{*-}_{\rho_k}$. 

For $C$ large, up to level $n$ there will be no branches of height greater than $C\log(n)$ with high probability (cf.\ \cite{bo18}). It follows that $X^*$ and the RTRW deviate by at most $C\log(n)$ with high probability.
Since we consider polynomial scaling, it suffices to consider this randomly trapped random walk instead of $X^*$ and thus $X$. To ease notation, we now denote by $X$ the randomly trapped random walk.

%Let $\omega=(\oT_x)_{x \in \Zb}$ denote the sequence of these trees. For $\omega$ fixed let $(\eta_{k,i})_{k \in \Zb,i\geq 0}$ be independent with 
%\[\Pt^\omega\left(\eta_{k,i}=m\right)=\Pt^{\oT_{\rho_k}}_{\rho_k}\left(\min\{n>0:W_n=\oR_{\rho_k}\}=m\right).\]
%We then consider the randomly trapped random walk $\hat{X}_n$ with these holding times.
%
%Using the excursion times and embedded walk $Y_n^*$ of $X^*_n$ in the above construction of a randomly trapped random walk, we have a coupling between the biased random walk $X^*_n$ and the RTRW $\hat{X}_n$ so that $\hat{X}_n$ is at $x \in \Zb$ precisely when $X^*_n$ is in the branch rooted at $\rho_x$. 

\subsection{A decomposition of the trapping times}\label{s:DTT}
By Theorem \ref{t:fluc} it suffices to show that for any $l \in\Nb$, $\lambda>0$ and a real valued function $f$ satisfying $f(1,\lambda)\asymp\lambda^{\gamma}$,
\begin{flalign*}
n_k(\varsigma)\log\left(\Er\left[\Et^\oT\left[\exp\left(-\frac{\lambda}{n_k(\varsigma)^{1/\gamma}}(\eta_0-\Eb[\eta_0])\right)\right]^l\right]\right)\sim f(l,\lambda)
\end{flalign*}
as $k \rightarrow \infty$. With a slight abuse of notation we will write $n$ instead of $n_k(\varsigma)$ until we need to consider the specific subsequences.

The purpose of this section is to decompose the holding times of the RTRW; we begin by proving upper bounds on the expected excursion times of biased random walks in trees. Let $\Ts$ be a fixed tree rooted at a vertex $\rho$ with germ $\oR$. Let $W$ be a random walk on $\Ts$ with transitions given by \eqref{e:W} and first return times $\tau^+_x$. Let $\oT$ be a random tree with germ $\oR$, root $\rho$ and first generation size $\xi^*-1$ which are the roots of independent GW-trees $(\Tc^\xi_k)_{k=1}^{\xi^*-1}$. Let $\oH$ denote the height of $\oT$. 
\begin{lem}\label{l:alpMom}
Suppose that $\Er[\xi^3]<\infty$, $\beta>1$ and $\mu<1$ then, for any $\alpha\in(1,2]$ there exist constants $C_1, C_2<\infty$ such that for any $h\geq 1$
\begin{enumerate}
\item
\[\Et^{\Ts}_\rho[(\tau^+_{\oR})^\alpha]\leq C_1\sum_{k=0}^\infty \sum_{j=0}^{\infty}Z_k^{\Ts}(Z_j^{\Ts})^2\beta^{j\alpha};\]
\item 
\[\Er\left[\Et^{\oT}_\rho[(\tau^+_{\oR})^\alpha]\big|\oH\leq h\right]\leq C_2h(\beta^\alpha\mu)^h.\]
\end{enumerate}
\begin{proof}
For $x$ in $\Ts$ let $v_x$ denote the number of visits to $x$ before reaching $\oR$. Using convexity and that $v_x$ is geometrically distributed we have that
\begin{flalign*}
\Et^{\Ts}_\rho[(\tau^+_{\oR})^\alpha] 
 = \Et^{\Ts}_\rho\left[\left(\sum_{x\in\Ts}v_x\right)^\alpha\right] 
 \leq \left(\sum_{x\in\Ts}1\right)\left(\sum_{x\in\Ts}\Et^{\Ts}_\rho\left[v_x^\alpha\right]\right) 
 \leq C\sum_{k=0}^\infty Z_k^{\Ts}\sum_{x\in\Ts}(\beta |c(x)|+1)^\alpha\beta^{\alpha|\rho-x|}. 
%& \leq C\sum_{k=0}^\infty Z_k^{\Ts}\sum_{j=0}^{\infty}(Z_j^{\Ts})^2\beta^{j\alpha} 
\end{flalign*}
Noting that $(\beta |c(x)|+1)^\alpha\leq (\beta |c(x)|+1)^2\leq C((Z_{|x|+1}^{\Ts})^2+(Z_{|x|}^{\Ts})^2)$ gives the first result.

Since $\xi$ has finite third moments and $\Pr(\oH\leq 1)$ is bounded below by a constant we have that, by \cite[Lemma 4.1]{bo18_1}, 
\[\Er[Z_k^{\oT}(Z_j^{\oT})^2|\oH\leq h]\leq \frac{\ind_{\{j,k\leq h\}}\Er[Z_k^{\oT}(Z_j^{\oT})^2]}{\Pr(\oH\leq h)}\leq C\mu^{k\lor j}\ind_{\{j,k\leq h\}}.\]
Substituting this into the previous equation gives second result.
\end{proof}
\end{lem}

Since we can consider a single excursion in an individual branch, we will simplify notation by studying the holding time $\eta_0$ as the excursion time of the random walk $W$ (whose transition probabilities are given in \eqref{e:W}) on the dummy tree $\oT$ with germ $\oR$ and root $\rho$ which has $\xi^*-1$ offspring $(\rho_k)_{k=1}^{\xi^*-1}$ which are roots of independent GW-trees $(\Tc^\xi_k)_{k=1}^{\xi^*-1}$. We write $\oT^\xi_k$ for the tree formed by attaching the root $\rho$ to the tree $\Tc^\xi_k$. Let 
\[N_k:=\sum_{n=1}^{\eta_0}\ind_{\{W_{n-1}=\rho,\;W_n=\rho_k\}}  \]
be the number of entrances into the tree $\Tc^\xi_k$. Letting $\tau^{0,k}=-1$ we have that $\tau^{j,k}:=\inf\{n>\tau^{j-1,k}:W_n=\rho, \; W_{n+1}=\rho_k\}$ is the start time of the $j\th$ excursion into $\Tc^\xi_k$ and $\tau^{j,k}_\rho:=\inf\{n>\tau^{j,k}:W_n=\rho\}$ is the end time of the excursion. It follows that
\[\eta_0=1+\sum_{k=1}^{\xi^*-1}\sum_{j=1}^{N_k}(\tau^{j,k}_\rho-\tau^{j,k})\]
where the $1$ corresponds to the final step from the root to the germ. Each trap $\Tc^\xi_k$ has a deepest vertex which we denote by $\delta_k$ (if there is more than one of equal depth then we choose one from the final generation uniformly at random). 
We can then write
\[G_{j,k}:=\sum_{n=\tau^{j,k}}^{\tau^{j,k}_\rho}\ind_{\{W_n=\delta_k\}}\]
to be the number of visits to $\delta_k$ on the $j\th$ excursion into $\Tc^\xi_k$. In particular, we write 
\begin{enumerate}[i)]
\item
$\tau^{0,j,k}_\delta:=0$ and $\tau^{i,j,k}_\delta:=\inf\{n>\tau^{i-1,j,k}_\delta:W_n=\delta_k\}$ to be the hitting times of $\delta_k$ on the $j\th$ excursion into $\Tc^\xi_k$;
\item 
$T_{i,j,k}:= \tau^{i+1,j,k}_\delta-\tau^{i,j,k}_\delta$ to be the duration of the $i\th$ excursion from $\delta_k$ to itself on the $j\th$ excursion into $\Tc^\xi_k$ for $i=1,...,G_{j,k}-1$;
\item 
$T_{j,k}:= (\tau^{j,k}_\rho-\tau^{G_{j,k}-1,j,k}_\delta+\tau^{1,j,k}_\delta-\tau^{j,k})\ind_{\{G_{j,k}>0\}}+(\tau^{j,k}_\rho-\tau^{j,k})\ind_{\{G_{j,k}=0\}}$ to be the remainder of the $j\th$ excursion into $\Tc^\xi_k$ (which consists of the first journey to the deepest point, the last journey from the deepest point and the entire excursion if the deepest point is not reached).
\end{enumerate}
Using the convention that $\sum_{i=1}^{-1}\lambda_i=0$, we then have
\[\eta_0= 1+\sum_{k=1}^{\xi^*-1}\sum_{j=1}^{N_k} \left(T_{j,k}+\sum_{i=1}^{G_{j,k}-1}T_{i,j,k}\right).\]

We now show that we can ignore the parts of the excursion given by $T_{j,k}$ thus allowing us to consider holding times of the form given in \eqref{e:Red1} below. The proof of this follows by comparing the excursions forming $T_{j,k}$ with excursion times in unconditioned GW-trees in which the bias along the spine (from the root to the deepest point) acts towards the root and the bias inside the foliage acts away from the spine (but we control the height of the foliage).
\begin{lem}\label{l:outback}
Under the assumptions of Theorem \ref{t:GWT}, there exists $\alpha>\gamma$ such that 
\[\Eb\left[\left(1+\sum_{k=1}^{\xi^*-1}\sum_{j=1}^{N_k}T_{j,k}\right)^\alpha\right]<\infty.\] 
\end{lem}
\begin{proof}
Write $T^{(1)}:= (\tau^{1,1}_\rho-\tau^{G_{1,1}-1,1,1}_\delta)\ind_{\{G_{1,1}>0\}}$ to be the time taken to return to the root from the deepest point, $T^{(2)}:= (\tau^{1,1,1}_\delta-\tau^{1,1})\ind_{\{G_{1,1}>0\}}$ to be the time taken to reach the deepest point from the root and $T^{(3)}:= (\tau^{1,1}_\rho-\tau^{1,1})\ind_{\{G_{1,1}=0\}}$ to be the duration of an excursion when it does not reach the deepest point. Then $T_{1,1}=T^{(1)}+T^{(2)}+T^{(3)}$. By convexity we have that 
\begin{flalign*}
\Eb\left[\left(1+\sum_{k=1}^{\xi^*-1}\sum_{j=1}^{N_k}T_{j,k}\right)^\alpha\right]
%& = \Eb\left[(\xi^*)^{\alpha}\left(\frac{1+\sum_{k=1}^{\xi^*-1}\sum_{j=1}^{N_k}T_{j,k}}{\xi^*}\right)^\alpha\right] \\
%& \leq \Eb\left[(\xi^*)^{\alpha-1}\left(1+\sum_{k=1}^{\xi^*-1}\left(\sum_{j=1}^{N_k}T_{j,k}\right)^\alpha\right)\right]\\
%& = \Eb\left[(\xi^*)^{\alpha-1}\left(1+\sum_{k=1}^{\xi^*-1}N_k^{\alpha}\left(\sum_{j=1}^{N_k}\frac{T_{j,k}}{N_k}\right)^\alpha\right)\right] \\
%& \leq \Eb\left[(\xi^*)^{\alpha-1}\left(1+\sum_{k=1}^{\xi^*-1}N_k^{\alpha-1}\sum_{j=1}^{N_k}T_{j,k}^\alpha\right)\right] \\
%& = \Eb\left[(\xi^*)^{\alpha-1}\right] +\Eb\left[(\xi^*)^{\alpha}\right]\Eb\left[N_1^\alpha\right]\Eb\left[T_{1,1}^\alpha\right].
& \leq C\Eb\left[(\xi^*)^{\alpha}\right]\Eb\left[N_1^\alpha\right]\left(\Eb\left[(T^{(1)})^\alpha\right]+\Eb\left[(T^{(2)})^\alpha\right]+\Eb\left[(T^{(3)})^\alpha\right]\right).
\end{flalign*} 
Since $\xi$ has finite third moments we have that $\xi^*$ has finite second moments and therefore $\Eb\left[(\xi^*)^{\alpha}\right]<\infty$ for $\alpha\leq 2$. The number of entrances into the first trap is geometrically distributed with termination probability $(\beta+1)/(2\beta+1)$ and therefore $\Eb\left[N_1^\alpha\right]<\infty$ for any $\alpha<\infty$. 

For convenience, we will consider a biased walk $W$ on a dummy trap $\Tc^\xi$ with root $\rho$ and deepest point $\delta$. Let $\Hc$ denote the height of $\Tc^\xi$ and $\Pr^h(\cdot)=\Pr(\cdot|\Hc=h)$ the environment law conditioned on the hight of the tree. Then
\begin{flalign*}
\Eb\left[(T^{(1)})^\alpha\right]
&\leq\sum_{h=1}^\infty \Pr(\Hc=h)\Er^h\left[\Et^{\Tc^\xi}_{\delta}\left[(\tau_\rho^+)^\alpha|\tau_\rho^+<\tau_\delta^+\right]\right].
%\Eb\left[(T^{(2)})^\alpha\right]
%&\leq\sum_{h=1}^\infty \Pr(\Hc=h)\Er^h\left[\Et^{\Tc^\xi}_{\rho}\left[(\tau_\delta^+)^\alpha|\tau_\delta^+<\tau_\rho^+\right]\right],\\
%\Eb\left[(T^{(3)})^\alpha\right]
%&\leq\sum_{h=1}^\infty \Pr(\Hc=h)\Er^h\left[\Et^{\Tc^\xi}_{\rho}\left[(\tau_\rho^+)^\alpha|\tau_\rho^+<\tau_\delta^+\right]\right].
\end{flalign*}

Recall that a trap $\Tc^\xi$ has a unique self-avoiding path connecting $\rho$ and $\delta$ which we call the spine and denote by $\delta^0=\delta,\delta^1,...,\delta^\Hc=\rho$. For a vertex $\delta^k$ on the spine, we refer to the vertices which are descendants of $\delta^k$ but not on the spine as being in the subtrap $\Tc^\xi_{\delta^k}$ at $\delta^k$. Conditioning on $\tau_\rho^+<\tau_\delta^+$ or $\tau_\delta^+<\tau_\rho^+$ does not change the transition probabilities inside the subtraps or the probability of moving into a subtrap (cf.\ \cite{arfrgaha12}). Let 
\[\Lc_k=\sum_{n=1}^{\tau_\delta^+\land\tau_\rho^+}(\ind_{\{W_{n-1}=\delta^{k-1}\}\cap\{W_n=\delta^k\}}+\ind_{\{W_{n-1}=\delta^{k+1}\}\cap\{W_n=\delta^k\}})\] 
denote the number of visits to $\delta^k$ from one its neighbours on the spine before reaching $\rho$ or $\delta$. Using convexity we then have that
\begin{flalign*}
\Eb\left[(T^{(1)})^\alpha\right]
& \leq \sum_{h=1}^\infty \Pr(\Hc=h)h^{\alpha-1}\sum_{k=1}^h\Er^h\left[\Et^{\Tc^\xi}_{\delta}\left[\Lc_k^\alpha|\tau_\rho^+<\tau_\delta^+\right]\right]\Er^h\left[\Et^{\Tc^\xi}_{\delta^k}\left[(\tau^+_{\delta^{k-1}}\land\tau^+_{\delta^{k+1}})^\alpha\right]\right].
%\Eb\left[(T^{(2)})^\alpha\right]
%& \leq \sum_{h=1}^\infty \Pb(\Hc=h)h^{\alpha-1}\sum_{k=1}^h\Er^h\left[\Et^{\oT^f}_{\rho}\left[\Lc_k^\alpha|\tau_\delta^+<\tau_\rho^+\right]\right]\Er^h\left[\Et^{\oT^f}_{\delta^k}\left[(\tau^+_{\delta^{k-1}}\land\tau^+_{\delta^{k+1}})^\alpha\right]\right], \\
%\Eb\left[(T^{(3)})^\alpha\right]
%& \leq \sum_{h=1}^\infty \Pb(\Hc=h)h^{\alpha-1}\sum_{k=1}^h\Er^h\left[\Et^{\oT^f}_{\rho}\left[\Lc_k^\alpha|\tau_\rho^+<\tau_\delta^+\right]\right]\Er^h\left[\Et^{\oT^f}_{\delta^k}\left[(\tau^+_{\delta^{k-1}}\land\tau^+_{\delta^{k+1}})^\alpha\right]\right]. \\
\end{flalign*}
Similarly, we have
\begin{flalign*}
\Eb\left[(T^{(2)})^\alpha\right]
& \leq \sum_{h=1}^\infty \Pr(\Hc=h)h^{\alpha-1}\sum_{k=1}^h\Er^h\left[\Et^{\Tc^\xi}_{\rho}\left[\Lc_k^\alpha|\tau_\delta^+<\tau_\rho^+\right]\right]\Er^h\left[\Et^{\Tc^\xi}_{\delta^k}\left[(\tau^+_{\delta^{k-1}}\land\tau^+_{\delta^{k+1}})^\alpha\right]\right], \\
\Eb\left[(T^{(3)})^\alpha\right]
& \leq \sum_{h=1}^\infty \Pr(\Hc=h)h^{\alpha-1}\sum_{k=1}^h\Er^h\left[\Et^{\Tc^\xi}_{\rho}\left[\Lc_k^\alpha|\tau_\rho^+<\tau_\delta^+\right]\right]\Er^h\left[\Et^{\Tc^\xi}_{\delta^k}\left[(\tau^+_{\delta^{k-1}}\land\tau^+_{\delta^{k+1}})^\alpha\right]\right]. 
\end{flalign*}
The only term that differs in these expressions is the expected local time at vertices on the spine. By the Gambler's ruin, the walk on the spine can be stochastically dominated by a biased random walk where the bias favours steps towards $\delta$ when conditioned on $\tau_\delta^+<\tau_\rho^+$ and $\rho$ when conditioned on $\tau_\rho^+<\tau_\delta^+$. It immediately follows that there exists a constant $C$ which is independent of $h$ and $k$ such that 
\[\Er^h\left[\Et^{\Tc^\xi}_{\delta}\left[\Lc_k^\alpha|\tau_\rho^+<\tau_\delta^+\right]\right],\Er^h\left[\Et^{\Tc^\xi}_{\rho}\left[\Lc_k^\alpha|\tau_\delta^+<\tau_\rho^+\right]\right],\Er^h\left[\Et^{\Tc^\xi}_{\rho}\left[\Lc_k^\alpha|\tau_\rho^+<\tau_\delta^+\right]\right]\leq C.\]

Since the number of children $\delta^k$ has off the spine is stochastically dominated by a size biased random variable and each of these is the root of a GW-tree conditioned to have height at most $k-1$, by Lemma \ref{l:alpMom},
\[\Er^h\left[\Et^{\Tc^\xi}_{\delta^k}\left[(\tau^+_{\delta^{k-1}}\land\tau^+_{\delta^{k+1}})^\alpha\right]\right]
\leq\Er\left[\Et^{\oT}_\rho[(\tau^+_{\oR})^\alpha]\big|\oH\leq k\right]
 \leq Ck(\beta^\alpha\mu)^{k}.\]
By \eqref{l:cmu} we have that $\Pr(\Hc=h)\leq C\mu^h$ for some constant $C<\infty$. It follows that for $i=1,2,3$
\[\Eb\left[(T^{(i)})^\alpha\right]\leq C\sum_{h=1}^\infty \mu^hh^{\alpha-1}\sum_{k=1}^hk(\beta^\alpha\mu)^k \leq C\sum_{h=1}^\infty h^{\alpha}(\beta^\alpha\mu^2)^h\]
which is finite since $\beta^\alpha \mu^2<1$.
\end{proof}

By Lemma \ref{l:outback} and Corollary \ref{c:ign} we can now consider holding times of the form
\begin{flalign}\label{e:Red1}
\sum_{k=1}^{\xi^*-1}\sum_{j=1}^{N_k} \sum_{i=1}^{G_{j,k}-1}T_{i,j,k}.
\end{flalign}

The excursions $T_{i,j,k}$ from the deepest point $\delta_k$ in the trap $\Tc^\xi_k$ are typically very short because the bias is acting towards $\delta_k$. They do, however, depend on the height of the trap which (through $G_{j,k}$) is largely responsible for the heavy tailed behaviour in the above sum. We therefore want to consider excursion times from the deepest point that do not depend on the height of the trap; for this, we will replace $T_{i,j,k}$ with excursion times on $\Tc^\xi_k$ extended to an infinite trap $\Tc_k^*$ using the construction of \cite{geke99}. That is, starting from $\Tc^\xi_k$, we can iteratively construct a random infinite tree $\Tc_k^*$ with deepest point $\delta_k$ such that, denoting by $\delta_k^1$ the parent of $\delta_k$ and $\delta_k^{l+1}$ the parent of $\delta_k^l$, 
\begin{enumerate}[1)]
\item
the subtree of $\Tc_k^*$ consisting of $\delta_k^l$ and all of its descendants is a GW-tree conditioned to have height $l$;
\item 
for $l=|\rho_k-\delta_k|$, the subtree of $\Tc_k^*$ consisting of $\delta_k^l$ and all of its descendants is $\Tc^\xi_k$.
\end{enumerate}

We now show that we can replace $T_{i,j,k}$ with excursion times on $\Tc^*_k$. We begin by constructing these excursion times.
%Let $X_n^1$ be a $\beta$-biased random walk on $\Tc_1^*$ started from $X_0^1=\delta_1$. Let $T_{1,1,1}^*=\inf\{n>0:X_n^1=\delta_1\}$ be the first return time to $\delta_1$ and $\Ac_{1,1,1}:=\{T_{1,1,1}^*>\inf\{n>0:X_n^1=\delta_1^{|\rho_1-\delta_1|}\}$ be the event that the walk leaves $\Tc^\xi_1$ before returning to $\delta_1$. We then have that $\tilde{T}_{1,1,1}:=T_{1,1,1}^*\ind_{\Ac_{1,1,1}}+T_{1,1,1}\ind_{\Ac_{1,1,1}^c}$ has the same law as $T_{1,1,1}^*$ but is coupled to $T_{1,1,1}$ such that $T_{1,1,1}=\tilde{T}_{1,1,1}+\ind_{\{\Ac_{1,1,1}\}}(T_{1,1,1}-T_{1,1,1}^*)$.
%Similarly, we can construct random variables $(\tilde{T}_{i,j,k})_{i,j,k\geq 1}$ such that,
%\begin{enumerate}
% \item
%for each $k$, $(\tilde{T}_{i,j,k})_{i,j\geq 1}$ are i.i.d\ with respect to $\Pt^{\Tc^*_k}$ and coupled to $T_{i,j,k}$ such that $T_{i,j,k}=\tilde{T}_{i,j,k}+\ind_{\{\Ac_{i,j,k}\}}(T_{i,j,k}-T_{i,j,k}^*)$ where $(T_{i,j,k}^*)_{i,j\geq 1}$ are independent excursion times on $\Tc_k^*$ and $\Ac_{i,j,k}$ are the events that the corresponding walks leave $\Tc_k^\xi$;
%\item 
%$(\tilde{T}_{1,1,k})_{k\geq 1}$ are i.i.d.\ with respect to $\Pb$.
% \end{enumerate} 
For $k= 1,...,\xi^*-1$ let
 \begin{enumerate}[1)]
 \item 
$X_n^k$ be independent $\beta$-biased random walks on $\Tc_k^*$ started from $X_0^k=\delta_k$;
\item 
$T_{1,1,k}^*=\inf\{n>0:X_n^k=\delta_k\}$ be the first return time to $\delta_k$; 
\item 
$\Ac_{1,1,k}:=\{T_{1,1,k}^*>\inf\{n>0:X_n^k=\delta_k^{|\rho_k-\delta_k|}\}$ be the event that the walk leaves $\Tc^\xi_k$ before returning to $\delta_k$. 
\end{enumerate}
We then have that 
\begin{enumerate}[1)]
\item
$\tilde{T}_{1,1,k}:=T_{1,1,k}^*\ind_{\Ac_{1,1,k}}+T_{1,1,k}\ind_{\Ac_{1,1,k}^c}$ has the same law as $T_{1,1,k}^*$ but is coupled to $T_{1,1,k}$ such that $T_{1,1,k}=\tilde{T}_{1,1,k}+\ind_{\{\Ac_{1,1,k}\}}(T_{1,1,k}-T_{1,1,k}^*)$;
\item 
$(\tilde{T}_{1,1,k})_{k\geq 1}$ are i.i.d.\ with respect to $\Pb$;
\item
for each $k$ we can construct $(\tilde{T}_{i,j,k})_{i,j\geq 1}$ which are
\begin{enumerate}[i)]
\item
i.i.d.\ with respect to $\Pt^{\Tc^*_k}$;
 \item 
coupled to $T_{i,j,k}$ such that $T_{i,j,k}=\tilde{T}_{i,j,k}+\ind_{\{\Ac_{i,j,k}\}}(T_{i,j,k}-T_{i,j,k}^*)$ where $(T_{i,j,k}^*)_{i,j\geq 1}$ are independent excursion times on $\Tc_k^*$ and $\Ac_{i,j,k}$ are the events that the corresponding walks leave $\Tc_k^\xi$.
 \end{enumerate}
 \end{enumerate} 
 
Before showing that we can replace $T_{i,j,k}$ with $\tilde{T}_{i,j,k}$, we first show that $\tilde{T}_{i,j,k}$ have finite second moments. 
 \begin{lem}\label{l:finVar}
Under the assumptions of Theorem \ref{t:GWT} we have that $\Eb[\tilde{T}_{1,1,1}^2]<\infty$.
 \end{lem}
\begin{proof}
Let $\delta^*$ denote the spine $\{\delta^0,\delta^1,...\}$ of $\Tc^*$ and, for $x\in\Tc^*$, let $v_x$ denote the number of visits to $x$ before returning to $\delta=\delta^0$. By the Gambler's ruin we have that $\Pt^{\Tc^*}_\delta(\tau_{\delta^k}^+<\tau_{\delta}^+)\leq C\beta^{-k}$.  Since the walk on the spine is $\beta$-biased we have that $\Et^{\delta^*}_{\delta^k}[(v_{\delta^k})^2]\leq C$ independently of $k$. Moreover, the time spent in a subtrap rooted at $\delta^k$ is stochastically dominated by the time spent in the tree $\oT$ conditioned to have height at most $k$. By Lemma \ref{l:alpMom} it follows that
\begin{flalign*}
\Eb[\tilde{T}_{1,1,1}^2] 
 \leq \sum_{k=0}^\infty \Er\left[\Pt^{\Tc^*}_\delta(\tau_{\delta^k}^+<\tau_{\delta}^+)\right]\Er\left[\Et^{\delta^*}_{\delta^k}[(v_{\delta^k})^2]\right]\Er\left[\Et^{\oT}_\rho[(\tau^+_{\oR})^2]|\oH\leq k\right] 
 \leq \sum_{k=0}^\infty\beta^{-k}k(\beta^2\mu)^k 
\end{flalign*}
which is finite since $\beta\mu<1$.
\end{proof}
 
The following lemma shows that we can replace the excursions $T_{i,j,k}$ with excursions on the infinite traps (which are independent of the height of the original trap).
\begin{lem}\label{l:RemCond}
Under the assumptions of Theorem \ref{t:GWT} we have that, for some $\alpha>\gamma$, 
\[\Eb\left[\left(\sum_{k=1}^{\xi^*-1}\sum_{j=1}^{N_k}\sum_{i=1}^{G_{j,k}-1}\ind_{\{\Ac_{i,j,k}\}}(T_{i,j,k}-T_{i,j,k}^*)\right)^{\!\!\!\alpha}\right]<\infty.\] 
 \end{lem} 
\begin{proof}
Since $\Er[(\xi^*)^\alpha],\Er[N_k^\alpha]<\infty$, by convexity it suffices to show that
\begin{flalign}\label{e:alpDif}
\Eb\left[\left(\sum_{i=1}^{G_{1,1}-1}\ind_{\{\Ac_{i,1,1}\}}(T_{i,1,1}-T_{i,1,1}^*)\right)^{\!\!\!\alpha}\right]
\end{flalign}
is finite. Summing over the height $\Hc$ of the trap $\Tc^\xi_1$ and using that $\Pr(\Hc=h)\leq C\mu^h$ we have that \eqref{e:alpDif} is bounded above by
\begin{flalign*}
%&\Eb\left[\left(\sum_{i=1}^{G_{1,1}-1}\ind_{\{\Ac_{1,1,k}\}}(T_{1,1,k}-T_{1,1,k}^*)\right)^\alpha\right]\\
& \sum_{h=1}^\infty \!\Pr(\Hc=h)\Eb^h\!\left[\!\left(\sum_{i=1}^{G_{1,1}-1}\ind_{\{\Ac_{i,1,1}\}}(T_{i,1,1}-T_{i,1,1}^*)\right)^{\!\!\!\alpha}\right] %\\
 \leq \sum_{h=1}^\infty \mu^h\Eb^h\!\left[G_{1,1}^\alpha\right] \Eb^h\!\left[\left(\ind_{\{\Ac_{1,1,1}\}}(T_{1,1,1}-T_{1,1,1}^*)\right)^{\alpha}\right].
\end{flalign*}

It follows from Lemma \ref{l:alpMom} that $\Eb\left[\left(T_{1,1,1}-T_{1,1,1}^*\right)^\alpha\big| \Ac_{1,1,1}, \Hc=h\right]\leq h(\beta^\alpha\mu)^h$. Conditional on $\Hc=h$ we have that $G_{1,1}$ is geometrically distributed with termination probability $(\beta^h-\beta)/(\beta^h-1)$ which gives us that $\Eb\left[G_{1,1}^\alpha\big| \Hc=h\right]\leq C\beta^{\alpha h}$. Similarly, $\Pb(\Ac_{1,1,1}|\Hc=h)\leq C\beta^{-h}$ therefore we have that \eqref{e:alpDif} is bounded above by
\begin{flalign*}
\sum_{h=1}^\infty \mu^h\beta^{h(\alpha-1)} \Eb\left[\left(T_{1,1,1}-T_{1,1,1}^*\right)^\alpha\big| \Ac_{1,1,1}, \Hc=h\right]\leq \sum_{h=1}^\infty h(\beta^{2\alpha-1}\mu^2)^h.
\end{flalign*}
Since $\alpha>\gamma=\log(\mu^{-1})/\log(\beta)$ we have that $\beta^{2\alpha-1}\mu^2=\beta^{2(\alpha-\gamma)-1}<1$ for $\alpha<\gamma+1/2$. For such $\alpha$ we then have that the above sum is finite which completes the proof.
%\[\sum_{h=1}^\infty h(\beta^{2\alpha-1}\mu^2)^h\leq\sum_{h=1}^\infty h(\beta^{2(\alpha-\gamma)-1})^h<\infty \]
\end{proof}
 
 By Lemma \ref{l:RemCond} and Corollary \ref{c:ign} we can now consider holding times of the form
\[\sum_{k=1}^{\xi^*-1}\sum_{j=1}^{N_k} \sum_{i=1}^{G_{j,k}-1}\tilde{T}_{i,j,k}.\]

On the event that $\delta_k$ is reached on the $j\th$ excursion into $\Tc_k^\xi$, the random variable $G_{j,k}-1$ is geometrically distributed with termination probability 
\[q_{\Hc_k}:=\Pt_{\delta_k}^{\Tc^\xi_k}(\tau_\rho^+<\tau_{\delta_k}^+)=\frac{\beta-1}{\beta^{\Hc_k}-1}\]
where $\Hc_k=|\rho-\delta_k|$ is the height of the trap $\Tc_k^\xi$. We will replace $N_k$ with the number of excursions into $\Tc_k^\xi$ such that $\delta_k$ is reached so that we can consider $G_{j,k}$ as geometrically distributed. The probability that $\delta_k$ is reached on a given excursion depends on the height of the trap; we also remove this dependency. 

Write $B_k:=\sum_{j=1}^{N_k}\ind_{\{G_{j,k}>0\}}$ for the number of excursions that reach $\delta_k$ then $B_k$ is binomially distributed with $N_k$ trials and success probability 
\[\Pt_{\rho_k}^{\oT_k^\xi}(\tau_{\delta_k}^+<\tau_{\rho}^+)=\frac{\beta^{\Hc_k}(1-\beta^{-1})}{\beta^{\Hc_k}-1}=(1-\beta^{-1})+\frac{1-\beta^{-1}}{\beta^{\Hc_k}-1}.\]

Let $\tilde{B}_k$ be binomially distributed with $B_k$ trials and success probability $\beta^{-\Hc_k}$. Then $\tilde{B}_k$ is binomially distributed with $N_k$ trials and success probability $1-\beta^{-1}$ such that $\tilde{B}_k\leq B_k$. Let $\tilde{G}_{j,k}$ be independent and equal in distribution (with respect to $\Pt^{\oT}$) to $G_{j,k}-1$ conditional on $\{G_{j,k}>0\}$.
\begin{lem}\label{l:Bin}
Under the assumptions of Theorem \ref{t:GWT} we have that, for some $\alpha>\gamma$, 
\[\Eb\left[\left(\sum_{k=1}^{\xi^*-1}\sum_{j=1}^{B_k-\tilde{B}_k}\sum_{i=1}^{\tilde{G}_{j,k}}\tilde{T}_{i,j,k}\right)^{\!\!\!\alpha}\right]<\infty.\] 
 \end{lem} 
\begin{proof}
Since $\Er[(\xi^*)^\alpha]<\infty$, by convexity it suffices to show that
\begin{flalign}\label{e:Bin}
\Eb\left[\left(\sum_{j=1}^{B_1-\tilde{B}_1}\sum_{i=1}^{\tilde{G}_{j,1}}\tilde{T}_{i,j,1}\right)^{\!\!\!\alpha}\right]<\infty
\end{flalign}
in finite.

Since $B_1-\tilde{B}_1\leq N_1\ind_{\{B_1\neq \tilde{B_1}\}}$ we have that \eqref{e:Bin} is bounded above by
\begin{flalign*}
\Eb\left[\left(\sum_{j=1}^{N_1}\ind_{\{B_1\neq \tilde{B_1}\}}\sum_{i=1}^{\tilde{G}_{j,1}}\tilde{T}_{i,j,1}\right)^{\!\!\!\alpha}\right] \leq \Eb[N_1^\alpha]\Eb\left[\left(\ind_{\{B_1\neq \tilde{B_1}\}}\sum_{i=1}^{\tilde{G}_{1,1}}\tilde{T}_{i,1,1}\right)^{\!\!\!\alpha}\right]
\end{flalign*}
where we have that $\Eb[N_1^\alpha]<\infty$. Summing over the height $\Hc$ of the trap $\Tc^\xi_1$, using that $\Pr(\Hc=h)\leq C\mu^h$ and independence of $\tilde{G}_{1,1}$ with $(\tilde{T}_{i,1,1})_{i\geq 1}$ we have that \eqref{e:Bin} is bounded above by
\begin{flalign*}
%\Eb\left[\left(\ind_{\{B_1\neq \tilde{B_1}\}}\sum_{i=1}^{\tilde{G}_{1,1}}\tilde{T}_{i,1,1}\right)^\alpha\right]
%& = \sum_{h=1}^\infty \Pb(\Hc=h)\Pb(B_1\neq \tilde{B}_1|\Hc=h)\Eb\left[\left(\sum_{i=1}^{\tilde{G}_{1,1}}\tilde{T}_{i,1,1}\right)^\alpha\big|B_1\neq \tilde{B_1}, \Hc=h\right] \\
 C\sum_{h=1}^\infty \mu^h \Pb(B_1\neq \tilde{B}_1|\Hc=h)\Eb[\tilde{G}_{1,1}^\alpha|\Hc=h]\Eb[\tilde{T}_{1,1,1}^\alpha].
\end{flalign*}
By definition of $B_k$ and $\tilde{B}_k$ we have that $\Pb(B_1\neq \tilde{B}_1|\Hc=h)\leq C \beta^{-h}$ and, as in Lemma \ref{l:RemCond}, we have that $\Eb[\tilde{G}_{1,1}^\alpha|\Hc=h]\leq \beta^{\alpha h}$. By Lemma \ref{l:alpMom} we have that $\Eb[\tilde{T}_{1,1,1}^\alpha]<\infty$ therefore it follows that \eqref{e:Bin} can be bounded above by 
\[C\sum_{h=1}^\infty (\beta^{\alpha-1}\mu)^h \]
which completes the proof since $\beta^{\alpha-1}\mu \in(0,1)$.
\end{proof}

By Corollary \ref{c:ign} we can now consider holding times of the form
\[\tilde{\eta}_0:=\sum_{k=1}^{\xi^*-1}\sum_{j=1}^{\tilde{B}_k} \sum_{i=1}^{\tilde{G}_{j,k}}\tilde{T}_{i,j,k}\]
where we recall that
\begin{enumerate}[1)]
\item
$\xi^*-1$ is the number of children of $\rho$ in $\oT$;
\item 
$\tilde{B}_k$ are jointly equal in distribution to the number of excursions into $\Tc_k^\xi$ which reach the deepest points $\delta_k$;
\item 
$\tilde{G}_{j,k}$ are equal in distribution to the number of excursions from $\delta_k$ to itself (for a walk started at $\delta_k$) which depend on $\oT$ only through $\Hc_k$;
\item 
$\tilde{T}_{i,j,k}$ are excursions (from $\delta_k$ to itself) in the infinite traps $\Tc_k^*$ and, conditional on $\Tc_k^*$, are independent of $\Hc_k$.
\end{enumerate}

\subsection{Subsequential convergence of the Laplace transforms}\label{s:ARTRW}
In this section we show that for any $l \in\Nb$, $\lambda>0$ and a real valued function $f$ satisfying $f(1,\lambda)\asymp\lambda^{\gamma}$,
\begin{flalign*}
n_k(\varsigma)\log\left(\Er\left[\Et^\oT\left[\exp\left(-\frac{\lambda}{n_k(\varsigma)^{1/\gamma}}(\tilde{\eta}_0-\Eb[\tilde{\eta}_0])\right)\right]^l\right]\right)\sim f(l,\lambda)
\end{flalign*}
as $k \rightarrow \infty$. As before, we continue to write $n$ for $n_k(\varsigma)$ until we require the specific subsequences. We begin by rearranging the quenched Laplace transform of $\tilde{\eta}_0$ using the distributions and dependence structure of $\xi^*$, $\tilde{B}_k$, $\tilde{G}_{j,k}$ and $T_{i,j,k}$ detailed in the previous section.
\begin{lem}\label{l:red}
Under the assumptions of Theorem \ref{t:GWT} we have that,
\[\Et^{\oT}\left[\exp\left(-\frac{\lambda}{n^{1/\gamma}}\tilde{\eta}_0\right)\right]= \frac{1}{1+\frac{\beta-1}{\beta+1}\sum_{k=1}^{\xi^*-1}\left(1-\Et^{\oT}\left[\exp\left(-\frac{\lambda}{n^{1/\gamma}}\sum_{i=1}^{\tilde{G}_{1,k}}\tilde{T}_{i,1,k}\right)\right]\right)}.\]
\end{lem}
\begin{proof}
Write $\vartheta_k:=\Et^{\oT}\left[\exp\left(-\frac{\lambda}{n^{1/\gamma}}\sum_{i=1}^{\tilde{G}_{1,k}}\tilde{T}_{i,1,k}\right)\right]$ then, since $\xi^*$ and $\vartheta_k$ are deterministic with respect to $\Pt^{\oT}$ and that $\tilde{T}_{i,j,k}$ and $\tilde{G}_{j,k}$ are independent of $\xi^*$ and $\tilde{B}_k$ we have that
\begin{flalign*}
\Et^{\oT}\left[\exp\left(-\frac{\lambda}{n^{1/\gamma}}\tilde{\eta}_0\right)\right] 
%& = \Et^{\oT}\left[\prod_{k=1}^{\xi^*-1}\prod_{j=1}^{\tilde{B}_k}\exp\left(-\frac{\lambda}{n^{1/\gamma}}\sum_{i=1}^{\tilde{G}_{j,k}}\tilde{T}_{i,j,k}\right)\right] \\
%& = \Et^{\oT}\left[\prod_{k=1}^{\xi^*-1}\prod_{j=1}^{\tilde{B}_k}\Et^{\oT,\xi^*,\tilde{B}}\left[\exp\left(-\frac{\lambda}{n^{1/\gamma}}\sum_{i=1}^{\tilde{G}_{j,k}}\tilde{T}_{i,j,k}\right)\right]\right] \\
 = \Et^{\oT}\left[\prod_{k=1}^{\xi^*-1}\vartheta_k^{\tilde{B}_k}\right] 
=\sum_{(b_k)_{k=1}^{\xi^*-1}\in \Zb_+^{\xi^*-1}}\left(\prod_{k=1}^{\xi^*-1}\vartheta_k^{b_k}\right)\Pt^{\oT}\left(\bigcap_{k=1}^{\xi^*-1}\tilde{B}_k=b_k\right). 
\end{flalign*}
Summing over the total number of excursions $\sum_{k=1}^{\xi^*-1}\tilde{B}_k$ we have that the above expression is equal to   
\begin{flalign*}
%&\Et^{\oT}\left[\prod_{k=1}^{\xi^*-1}\Et^{\oT}\left[\exp\left(-\frac{\lambda}{n^{1/\gamma}}\sum_{i=1}^{\tilde{G}_{1,k}}\tilde{T}_{i,1,k}\right)\right]^{\tilde{B}_k}\right] \\
%&  \sum_{(b_k)_{k=1}^{\xi^*-1}\in \Zb_+^{\xi^*-1}}\left(\prod_{k=1}^{\xi^*-1}\vartheta_k^{b_k}\right)\Pt^{\oT}\left(\bigcap_{k=1}^{\xi^*-1}\tilde{B}_k=b_k\right) \\
%& \qqquad= \sum_{(b_k)_{k=1}^{\xi^*-1}\in \Zb_+^{\xi^*-1}}\left(\prod_{k=1}^{\xi^*-1}\vartheta_k^{b_k}\right)\left(\frac{(\xi^*-1)(\beta-1)}{\xi^*(\beta-1)+2}\right)^{\sum_{k=1}^{\xi^*-1}b_k}\left(\frac{\beta+1}{\xi^*(\beta-1)+2}\right)(\xi^*-1)^{\sum_{k=1}^{\xi^*-1}b_k}\binom{\sum_{k=1}^{\xi^*-1}b_k}{b_1,...,b_{\xi^*-1}} \\
& \sum_{L=0}^\infty \left(\frac{(\xi^*-1)(\beta-1)}{\xi^*(\beta-1)+2}\right)^L\left(\frac{\beta+1}{\xi^*(\beta-1)+2}\right)\sum_{\substack{(b_k)_{k=1}^{\xi^*-1}\in \Zb_+^{\xi^*-1}\\ \sum_{k=1}^{\xi^*-1}b_k=L}}(\xi^*-1)^{-L}\binom{L}{b_1,...,b_{\xi^*-1}}\prod_{k=1}^{\xi^*-1}\vartheta_k^{b_k} \\
%& \qqquad= \frac{\beta+1}{\xi^*(\beta-1)+2}\sum_{L=0}^\infty \left(\frac{\beta-1}{\xi^*(\beta-1)+2}\right)^L\sum_{\substack{(b_k)_{k=1}^{\xi^*-1}\in \Zb_+^{\xi^*-1}\\ \sum_{k=1}^{\xi^*-1}b_k=L}}\binom{L}{b_1,...,b_{\xi^*-1}}\prod_{k=1}^{\xi^*-1}\vartheta_k^{b_k} \\
& \qqquad= \frac{\beta+1}{\xi^*(\beta-1)+2}\sum_{L=0}^\infty \left(\frac{(\beta-1)\sum_{k=1}^{\xi^*-1}\vartheta_k}{\xi^*(\beta-1)+2}\right)^L \\
%& \qqquad= \frac{\beta+1}{\xi^*(\beta-1)+2}\frac{1}{1-\frac{(\beta-1)\sum_{k=1}^{\xi^*-1}\vartheta_k}{\xi^*(\beta-1)+2}} \\
& \qqquad= \frac{1}{1+\frac{\beta-1}{\beta+1}\sum_{k=1}^{\xi^*-1}(1-\vartheta_k)}
\end{flalign*}
which completes the proof.
\end{proof}

The excursion times $(T_{i,1,k})_{i\geq 1}$ are identically distributed with respect to $\Pt^{\oT}$ and independent from $\tilde{G}_{1,k}$; therefore,
\begin{flalign*}
\Et^{\oT}\left[\exp\left(-\frac{\lambda}{n^{1/\gamma}}\sum_{i=1}^{\tilde{G}_{1,k}}\tilde{T}_{i,1,k}\right)\right]
& = \Et^{\oT}\left[\prod_{i=1}^{\tilde{G}_{1,k}}\exp\left(-\frac{\lambda}{n^{1/\gamma}}\tilde{T}_{i,1,k}\right)\right] \\
& = \Et^{\oT}\left[\Et^{\oT}\left[\exp\left(-\frac{\lambda}{n^{1/\gamma}}\tilde{T}_{1,1,k}\right)\right]^{\tilde{G}_{1,k}}\right] \\
%& = \frac{q_{\Hc_k}}{1-(1-q_{\Hc_k})\Et^{\oT}\left[\exp\left(-\frac{\lambda}{n^{1/\gamma}}\tilde{T}_{1,1,k}\right)\right]} \\
& = \frac{1}{1+\frac{1-q_{\Hc_k}}{q_{\Hc_k}}\Et^{\oT}\left[1-\exp\left(-\frac{\lambda}{n^{1/\gamma}}\tilde{T}_{1,1,k}\right)\right]}
\end{flalign*}
by using the probability generating function of a geometric random variable.

Write 
\[\psi_n^{(k)}=\frac{1-q_{\Hc_k}}{q_{\Hc_k}}\frac{\lambda}{n^{1/\gamma}}\Et^{\oT}[\tilde{T}_{1,1,k}]=\frac{\beta^{\Hc_k}-\beta}{\beta-1}\frac{\lambda}{n^{1/\gamma}}\Et^{\oT}[\tilde{T}_{1,1,k}]\]
then using Lemma \ref{l:finVar} and a Taylor expansion we have that 
\[\Er\left[\Et^{\oT}\left[\exp\left(-\frac{\lambda}{n^{1/\gamma}}\tilde{\eta}_0\right)\right]^l\right]
=\Er\left[\left(\frac{1}{1+\frac{\beta-1}{\beta+1}\sum_{k=1}^{\xi^*-1}\left(1-\frac{1}{1+\psi_n^{(k)}}\right)}\right)^l\right]+O(n^{-2/\gamma}).\]
This reduces the proof of Theorem \ref{t:GWT} to showing convergence of 
\begin{flalign}\label{e:Red2}
n\log\left(\Er\left[\left(\frac{1}{1+\frac{\beta-1}{\beta+1}\sum_{k=1}^{\xi^*-1}\left(1-\frac{1}{1+\psi_n^{(k)}}\right)}\right)^l\right]\exp\left(\frac{l\lambda\Eb[\tilde{\eta}_0]}{n^{1/\gamma}}\right)\right)
\end{flalign}
along the given subsequences. 

Let 
\[\Phi_n^{(M)}:=\frac{\beta-1}{\beta+1}\sum_{k=1}^{M}\left(1-\frac{1}{1+\psi_n^{(k)}}\right)\quad \text{then} \quad \Er[\Phi_n^{(\xi^*-1)}]=\frac{\lambda\Eb[\tilde{\eta}_0]}{n^{1/\gamma}}+\frac{\beta-1}{\beta+1}\Er[\xi^*-1]\Er\left[\frac{(\psi_n^{(1)})^2}{1+\psi_n^{(1)}}\right]\]
and the first expectation in \eqref{e:Red2} can be written as
\[1-l\left(\Er[\Phi_n^{(\xi^*-1)}]+\Er\left[\frac{(\Phi_n^{(\xi^*-1)})^2}{1+\Phi_n^{(\xi^*-1)}}\right]\right)+\sum_{j=2}^l(-1)^j\binom{l}{j}\Er\left[\left(\frac{\Phi_n^{(\xi^*-1)}}{1+\Phi_n^{(\xi^*-1)}}\right)^j\right].\]
Noting that 
\[\exp\left(\frac{l\lambda\Eb[\tilde{\eta}_0]}{n^{1/\gamma}}\right)=1+l\frac{\lambda\Eb[\tilde{\eta}_0]}{n^{1/\gamma}}+O(n^{-2/\gamma})=1+l\Er[\Phi_n^{(\xi^*-1)}]+O(n^{-2/\gamma}),\]
it suffices to show convergence of
%\[n\log\left(\left(1-l\left(\Er[\Phi_n^{(\xi^*-1)}]+\Er\left[\frac{(\Phi_n^{(\xi^*-1)})^2}{1+\Phi_n^{(\xi^*-1)}}\right]\right)+\sum_{j=2}^l(-1)^j\binom{l}{j}\Er\left[\left(\frac{\Phi_n^{(\xi^*-1)}}{1+\Phi_n^{(\xi^*-1)}}\right)^j\right]\right)\left(1+l\frac{\lambda\Eb[\tilde{\eta}_0]}{n^{1/\gamma}}+O(n^{-2/\gamma})\right)\right),\]
%for which, it suffices to show convergence of
\[n\Er\left[\frac{(\psi_n^{(1)})^2}{1+\psi_n^{(1)}}\right], \quad n\Er\left[\frac{(\Phi_n^{(\xi^*-1)})^2}{1+\Phi_n^{(\xi^*-1)}}\right] \quad \text{and} \quad n\Er\left[\left(\frac{\Phi_n^{(\xi^*-1)}}{1+\Phi_n^{(\xi^*-1)}}\right)^j\right]\]
for every $j\geq 2$. Using integration by parts we have that
\begin{flalign}\label{e:IBP}
n\Er\left[\frac{(\psi_n^{(1)})^2}{1+\psi_n^{(1)}}\right] & =n\int_0^\infty \frac{t(2+t)}{(1+t)^2}\Pr(\psi_n^{(1)}>t)\d t \\
n\Er\left[\frac{(\Phi_n^{(\xi^*-1)})^2}{1+\Phi_n^{(\xi^*-1)}}\right] 
& = n\sum_{M=1}^\infty\Pr(\xi^*-1=M)\Er\left[\frac{(\Phi_n^{(M)})^2}{1+\Phi_n^{(M)}}\right]\notag\\
 & = n\sum_{M=1}^\infty\Pr(\xi^*-1=M)\int_0^\infty \frac{t(2+t)}{(1+t)^2}\Pr(\Phi_n^{(M)}>t)\d t \notag\\
n\Er\left[\left(\frac{\Phi_n^{(\xi^*-1)}}{1+\Phi_n^{(\xi^*-1)}}\right)^j\right] 
& = n\sum_{M=1}^\infty\Pr(\xi^*-1=M)\Er\left[\left(\frac{\Phi_n^{(M)}}{1+\Phi_n^{(M)}}\right)^j\right]\notag \\
& = n\sum_{M=1}^\infty\Pr(\xi^*-1=M)\int_0^\infty \frac{jt^{j-1}}{(1+t)^{j+1}}\Pr(\Phi_n^{(M)}>t)\d t. \notag
\end{flalign}
Noting that $\xi^*-1$ has finite $1+\gamma$ moments and
\[\int_0^\infty \frac{t(2+t)}{(1+t)^2}t^{-\gamma}\d t, \; \int_0^\infty \frac{jt^{j-1}}{(1+t)^{j+1}}t^{-\gamma}\d t <\infty,\]
Lemma \ref{l:UpBd} shows that the integrands in the right hand side of \eqref{e:IBP} are bounded above by an integrable function. By dominated convergence, it follows that it remains to show convergence of $n\Pr(\psi_n^{(1)}>t)$ and $n\Pr(\Phi_n^{(M)}>t)$ for almost every $t>0$.
\begin{lem}\label{l:UpBd}
Under the assumptions of Theorem \ref{t:GWT}, there exists a constant $C$ such that for any $M\in\Nb$ we have
\[n\Pr(\psi_n^{(1)}>t)\leq Ct^{-\gamma} \quad \text{and} \quad n\Pr(\Phi_n^{(M)}>t)\leq  CM^{1+\gamma}t^{-\gamma}\]
uniformly over $n\in\Nb$ and $t>0$.
\end{lem}
\begin{proof}
Using that $\Pr(\Hc\geq t)\leq C\mu^t$ for some constant $C$ we have that
\begin{flalign*}
n\Pr(\psi^{(1)}>t) & \leq n\Pr\left(\beta^{\Hc}>\frac{C_{\lambda,\beta} t n^{1/\gamma}}{\Et^{\oT}[\tilde{T}_{1,1,1}]}\right) \\
&\leq n\Er\left[\Pr\left(\Hc>\frac{\log\left(C_{\lambda,\beta} t n^{1/\gamma}\right)-\log\left(\Et^{\oT}[\tilde{T}_{1,1,1}]\right)}{\log(\beta)}\big| \Et^{\oT}[\tilde{T}_{1,1,1}]\right)\right] \\
& \leq n \Er\left[\mu^{\frac{\log\left(C_{\lambda,\beta} t n^{1/\gamma}\right)-\log\left(\Et^{\oT}[\tilde{T}_{1,1,1}]\right)}{\log(\beta)}}\right] \\
& = Ct^{-\gamma} \Er\left[\Et^{\oT}[\tilde{T}_{1,1,1}]^{\gamma}\right] 
\end{flalign*}
which is bounded above by $Ct^{-\gamma}$ since $\Er\left[\Et^{\oT}[\tilde{T}_{1,1,1}]^2\right]<\infty$ by Lemma \ref{l:finVar}. Noting that $\Phi_n^{(M)}\leq C_\beta\sum_{k=1}^M\psi_n^{(k)}$, a union bound and the above estimate on $n\Pr(\psi^{(1)}>t)$ immediately give the bound for $n\Pr(\Phi_n^{(M)}>t)$.

%For $\Phi_n$ we have that
%\begin{flalign*}
%n\Pr(\Phi_n>t) & \leq n\Pr\left(\frac{\beta-1}{\beta+1}\sum_{j=1}^{\xi^*-1}\psi^{(j)}>t\right) \\
%& =n\sum_{l=1}^\infty\Pr(\xi^*-1=l)\Pr\left(\sum_{j=1}^l\psi^{(j)}>t\frac{\beta+1}{\beta-1}\right) \\
%& \leq n\sum_{l=1}^\infty\Pr(\xi^*-1=l)l\Pr\left(\psi^{(1)}>\frac{t}{l}\frac{\beta+1}{\beta-1}\right) \\
%& \leq C\sum_{l=1}^\infty\Pr(\xi^*-1=l)l^{1+\gamma}t^{-\gamma} 
%\end{flalign*}
%which is bounded above by $Ct^{-\gamma}$ since $\Er\left[(\xi^*)^{1+\gamma}\right]<\infty$.
\end{proof}

We now conclude by proving convergence of $n\Pr(\psi_n^{(1)}>t)$ and $n\Pr(\Phi_n^{(M)}>t)$ for almost every $t>0$. 
\begin{lem}\label{l:SubseqCon}
Under the assumptions of Theorem \ref{t:GWT} we have that, for every $M\in\Nb$, $n_k(\varsigma)\Pr(\psi_{n_k(\varsigma)}^{(1)}>t)$ and  $n_k(\varsigma)\Pr(\Phi_{n_k(\varsigma)}^{(M)}>t)$ converge as $k\rightarrow \infty$ for almost every $t>0$.
\end{lem}

\begin{proof}
Using that $\Pr(\Hc\geq m)\sim c_\mu\mu^m$ and $\Er[\Et^{\oT}[\tilde{T}_{1,1,1}]^\gamma]<\infty$ we have that
\begin{flalign*}
n_k\Pr(\psi^{(1)}_{n_k}>t) 
=  n_k\Pr\left(\beta^{\Hc}>\frac{t n_k^{1/\gamma}(\beta-1)}{\lambda\Et^{\oT}[\tilde{T}_{1,1,1}]}+\beta\right) 
= n_k\Er\left[\Pr\left(\Hc>\frac{\log\left(\frac{t n_k^{1/\gamma}(\beta-1)}{\lambda\Et^{\oT}[\tilde{T}_{1,1,1}]}+\beta\right)}{\log(\beta)}\Big| \Et^{\oT}[\tilde{T}_{1,1,1}]\right)\right] 
\end{flalign*}
converges for almost every $t>0$ due to our choice of subsequence $n_k=\lfloor\varsigma\mu^{-k}\rfloor$.

Let $\chi_k^{(i)}=\psi_{n_k}^{(i)}/(1+\psi_{n_k}^{(i)})$ then by the first part of the lemma we have that $n_k\Pr(\chi_{k}^{(i)}>t)$ converges for almost every $t>0$. In particular, $n_k\Pr(\Phi_n^{(1)}>t)$ converges for almost every $t>0$. To show that $n_k\Pr(\Phi_{n_k}^{(M)}>t)$ converges for almost every $t>0$ we show that 
\begin{flalign}\label{e:suml}
\lim_{k\rightarrow \infty}n_k(\Pr(\Phi_n^{(M)}>t)-M\Pr(\Phi_n^{(1)}>t))=0.
\end{flalign}
We show that \eqref{e:suml} holds for $M=2$, convergence then holds for general $M$ by an inductive argument.

For any $t>0$ we have that 
\begin{flalign*}
\Pr(\chi_k^{(1)}+\chi_k^{(2)}>t)  \geq  \Pr(\chi_k^{(1)}>t)+\Pr(\chi_k^{(2)}>t)-\Pr(\chi_k^{(1)}>t,\chi_k^{(2)}>t).
\end{flalign*}
By independence of $\chi_k^{(1)}$ and $\chi_k^{(2)}$ we have that $n_k\Pr(\chi_k^{(1)}>t,\chi_k^{(2)}>t)=n_k\Pr(\chi_k^{(1)}>t)\Pr(\chi_k^{(2)}>t)$ converges to $0$ as $k\rightarrow \infty$ thus $\liminf_{k\rightarrow\infty}n_k(\Pr(\Phi_n^{(2)}>t)-2\Pr(\Phi_n^{(1)}>t))\geq 0$. 

Next, we have that for any $\varepsilon>0$,
\begin{flalign*}
\Pr(\chi_k^{(1)}+\chi_k^{(2)}>t) %& \leq 1-\Pr(\chi_k^{(1)}<t-\varepsilon,\chi_k^{(2)}<\varepsilon)-\Pr(\chi_k^{(2)}<t-\varepsilon,
& \leq  \Pr(\chi_k^{(1)}>\varepsilon,\chi_k^{(2)}>\varepsilon)+\Pr(\chi_k^{(1)}>t-\varepsilon,\chi_k^{(2)}<\varepsilon)+\Pr(\chi_k^{(2)}>t-\varepsilon,\chi_k^{(1)}<\varepsilon). 
\end{flalign*}
By the previous part of the lemma we have that $n_k\Pr(\chi_k^{(1)}>\varepsilon,\chi_k^{(2)}>\varepsilon)$ converges to $0$ as $k\rightarrow \infty$. Moreover, by independence of $\chi_k^{(1)}$ and $\chi_k^{(2)}$ for almost every $t>0$ we have that 
\[\lim_{\varepsilon\rightarrow 0}\limsup_{k\rightarrow \infty}n_k|\Pr(\chi_k^{(1)}>t-\varepsilon)\Pr(\chi_k^{(2)}<\varepsilon) -\Pr(\chi_k^{(1)}>t)|=0\]
 therefore 
$\limsup_{k\rightarrow\infty}n_k(\Pr(\Phi_n^{(2)}>t)-2\Pr(\Phi_n^{(1)}>t))\leq 0$ which completes the proof.
\end{proof}

\section*{Acknowledgements}
I would like to thank David Croydon for his comments and many useful discussions. This work is supported by NUS grant R-146-000-260-114.

\end{document}